\documentclass[12pt, a4paper]{amsart}
\usepackage[margin=2.5cm]{geometry}

\usepackage{amstext,amssymb,amsmath,amsbsy}
\usepackage{amsthm}
\usepackage{setspace}
\usepackage[pageanchor]{hyperref}
\usepackage{indentfirst}
\usepackage{verbatim}
\usepackage{parskip}
\usepackage{enumerate}
\usepackage{graphicx}
\usepackage[dvipsnames]{xcolor}
\usepackage[OT1]{fontenc}
\usepackage[latin1]{inputenc}
\usepackage[english]{babel}

\newtheorem{theorem}{Theorem}[section]
\newtheorem{lemma}[theorem]{Lemma}

\newtheorem{proposition}[theorem]{Proposition}
\newtheorem{definition}[theorem]{Definition}

\newtheorem{remark}[theorem]{Remark}

\usepackage{mathtools}
\usepackage[ruled]{algorithm2e}
\usepackage[nameinlink]{cleveref}
\Crefname{algocf}{Algorithm}{Algorithms}

\begin{document}
\title[Non-Cartesian polynomials in product spaces]{On the Structure and Generic Non-Cartesianity of Polynomials in Product Spaces}
\date{}

\author[Shen]{Chun-Yen Shen}
\address{Department of Mathematics, National Taiwan University, Taiwan}
\email{cyshen@math.ntu.edu.tw}

\author[Truong]{Tuyen Trung Truong}
\address{Department of Mathematics, University of Oslo, Norway}
\email{tuyentt@math.uio.no}

\author[Yu]{Wei-Hsuan Yu}
\address{Department of Mathematics, National Central University, Taiwan}
\email{whyu@math.ncu.edu.tw}

\begin{abstract}
We develop a general theory of Cartesian and non-Cartesian polynomials on products of complex spaces 
$\mathbb{C}^{n_1} \times \cdots \times \mathbb{C}^{n_k}$. 
We prove that, for any fixed degree $d \ge 2$, a (Zariski) generic polynomial is non-Cartesian in a broad range of dimensions, establishing that Cartesian structure is highly exceptional.

We further introduce effective sufficient criteria for a polynomial to be non-Cartesian. Moreover, we  show that being (non)-Catersian can be decided algorithmically via Gr\"obner basis methods and quantitative forms of Hilbert's Nullstellensatz.

As an application, we connect the non-Cartesian condition to incidence geometry, obtaining sharp intersection bounds and constructing extremal configurations that demonstrate the optimality of these estimates.
\end{abstract}

\maketitle

\section{Introduction}

The study of zero sets of polynomials lies at the heart of many areas of mathematics, including complex analysis, algebraic geometry, dynamical systems, and combinatorics. From both geometric and analytic perspectives, polynomial zero sets encode deep structural information and frequently serve as the fundamental objects governing a problem. Foundational tools from commutative algebra and algebraic geometry, such as intersection theory and Gr\"obner bases, play an essential role in understanding these objects and their interactions \cite{Hartshorne,CoxLittleOShea,Eisenbud}.

While these questions are inherently continuous in nature, many applications require understanding how algebraic varieties interact with discrete sets. In particular, one is often led to study incidences between the zero set of a polynomial and Cartesian products or other finite configurations. Such problems arise naturally in combinatorial geometry, additive combinatorics, and related areas, where one seeks quantitative bounds on the size of intersections between algebraic sets and discrete structures. Classical incidence results of Pach and Sharir \cite{PachSharir}, together with later developments based on polynomial methods and partitioning techniques \cite{GuthKatz,KaplanMatousekSharir}, have played a fundamental role in shaping this subject. Related questions concerning incidences on Cartesian products and algebraic degeneracies have also been extensively studied in the context of the Elekes--Szab\'o theory and its later developments \cite{RazSharirZeeuw,SolymosiZeeuw}.

A central theme in modern mathematics is that algebraic structure strongly influences combinatorial growth. This philosophy appears prominently in the study of the sum--product phenomenon and its many geometric and arithmetic manifestations \cite{TaoVu}. Broadly speaking, highly structured or algebraically degenerate objects tend to admit unexpectedly large intersections and concentration phenomena, whereas generic or non-degenerate objects exhibit strong expansion properties and satisfy nontrivial incidence bounds. Important early manifestations of this principle appear in the work of Elekes and R\'onyai \cite{ElekesRonyai}, as well as in later developments concerning rational functions and sum--product estimates \cite{BukhTsimerman}. Closely related higher-dimensional incidence phenomena have also been investigated in a variety of settings; see for example \cite{SolymosiTao,SharirSolomon,Zahl,SharirZahl}.

One particularly striking manifestation of this principle is the dichotomy between Cartesian and non-Cartesian polynomials on $\mathbb{C}^2\times \mathbb{C}^2$ \cite{MojarradPhamValculescuZeeuw}, recalled in Definition 1.1 below. Cartesian structure represents a special algebraic degeneracy that permits large intersections with product sets, while non-Cartesian behavior forces substantially stronger incidence estimates. Understanding this distinction has become increasingly important in recent developments in incidence geometry and additive combinatorics, where it serves as a precise algebraic mechanism governing extremal configurations and sharp bounds.


\begin{definition}
A polynomial $P(x,y,z,t) \in \mathbb{C}[x,y,z,t]$ is said to be \emph{Cartesian} if it can be written in the form
\[
P(x,y,z,t)
= G(x,y)\,F_1(x,y,z,t) + H(z,t)\,F_2(x,y,z,t),
\]
where $G(x,y)$ and $H(z,t)$ are non-constant polynomials, and $F_1, F_2 \in \mathbb{C}[x,y,z,t]$.
\end{definition}

\begin{remark}
Since any polynomial in the variables $(x,y)$ (and similarly in $(z,t)$) factors into irreducible components in $\mathbb{C}[x,y,z,t]$ that still depend only on $(x,y)$, we may, without loss of generality, assume that both $G(x,y)$ and $H(z,t)$ are irreducible.
\end{remark}

The first question which would come to mind is: Can we design an algorithm to determine whether a given polynomial
\[
P(x,y,z,t) \in \mathbb{C}[x,y,z,t]
\]
is Cartesian? This will be addressed in Section 2. 

The notion of Cartesian polynomials can be extended to higher dimensions, please see later for detail. On the other hand, in lower dimensions the notion is not very useful. For example, we now show that all non-constant polynomials in dimension $2$ are Cartesian. Here, a polynomial $P(x,y) \in \mathbb{C}[x,y]$ is said to be \emph{Cartesian} if it can be written in the form
\[
P(x,y)
= G_1(x)\,H_1(x,y) + G_2(y)\,H_2(x,y),
\]
where $G_1(x)$ and $G_2(y)$ are non-constant polynomials.

\begin{lemma}
If $P(x,y)$ is non-constant, then it is Cartesian.
\end{lemma}

\begin{proof}
Without loss of generality, assume that $P$ depends nontrivially on $x$ (the argument is symmetric if it depends only on $y$).

Choose a generic $\alpha \in \mathbb{C}$ such that $P(\alpha,y)$ is not identically zero. Define
\[
G_1(x) = x - \alpha, \qquad H_2(x,y) = 1,
\]
and
\[
G_2(y) = P(\alpha,y).
\]
Then we can write
\[
P(x,y)
= (x-\alpha)\,H_1(x,y) + P(\alpha,y),
\]
where
\[
H_1(x,y) = \frac{P(x,y) - P(\alpha,y)}{x - \alpha}.
\]
Since $P(x,y) - P(\alpha,y)$ vanishes at $x=\alpha$, it is divisible by $(x-\alpha)$, so $H_1(x,y)$ is a polynomial.

Moreover, by the generic choice of $\alpha$, both $G_1(x)$ and $G_2(y)=P(\alpha,y)$ are non-constant. Hence $P$ is Cartesian.
\end{proof}

This notion was introduced and systematically studied in connection with incidence geometry, where it was shown that Cartesian structure is the only obstruction to obtaining nontrivial upper bounds on the size of intersections of algebraic varieties with Cartesian products of finite sets. More precisely, in the work of \cite{MojarradPhamValculescuZeeuw} and related developments, it was established that if $F(x,y,z,t)$ is non-Cartesian, then for finite sets $P,Q \subset \mathbb{C}^2$, one has the sharp bound
\[
|Z(F) \cap (P \times Q)| \lesssim |P|^{2/3} |Q|^{2/3} + |P| + |Q|,
\]
while Cartesian polynomials admit configurations with much larger intersections. 
This dichotomy reveals that Cartesian structure plays a fundamental role as the precise algebraic obstruction governing incidence phenomena.

Despite its importance, several basic questions about Cartesian polynomials remain poorly understood. 
In particular, it is natural to ask:

\begin{itemize}
\item How common are Cartesian polynomials among all polynomials of a fixed degree?
\item Can one effectively determine whether a given polynomial is Cartesian?
\item To what extent do these phenomena persist in higher-dimensional product spaces?
\end{itemize}

The purpose of this paper is to develop a systematic theory addressing these questions. 
Our results show that, in a precise sense, Cartesian structure is highly exceptional, and that non-Cartesian behavior is the generic situation.

\medskip

\noindent
\textbf{Generic non-Cartesianity.}
Our first main result establishes that non-Cartesian polynomials are ubiquitous.

\begin{theorem}[Generic non-Cartesianity]
\label{thm:generic}
Let $d \ge 2$. Then a Zariski generic polynomial 
\[
F \in \mathbb{C}[x,y,z,t]
\]
of degree at most $d$ is non-Cartesian.
\end{theorem}

We further extend this phenomenon to product spaces of higher dimension. The below is a natural higher-dimensional version of the Cartesian condition.

\begin{definition}[Cartesian polynomials on products]
Let
\[
X_j=\mathbb C^{n_j}, \qquad j=1,\dots,k,
\]
and write
\[
X=X_1\times\cdots\times X_k.
\]
For each \(j\), let
\[
x_j=(x_{j,1},\dots,x_{j,n_j})
\]
denote the coordinates on \(X_j\).

A polynomial
\[
F\in \mathbb C[x_1,\dots,x_k]
\]
is called \emph{Cartesian} with respect to the product decomposition
\[
\mathbb C^{n_1}\times\cdots\times \mathbb C^{n_k}
\]
if there exist non-constant polynomials
\[
G_j\in \mathbb C[x_j], \qquad j=1,\dots,k,
\]
and arbitrary polynomials
\[
H_j\in \mathbb C[x_1,\dots,x_k], \qquad j=1,\dots,k,
\]
such that
\[
F=\sum_{j=1}^k G_j(x_j)H_j(x_1,\dots,x_k).
\]
Equivalently,
\[
F\in (G_1(x_1),G_2(x_2),\dots,G_k(x_k))
\subset \mathbb C[x_1,\dots,x_k],
\]
where each \(G_j\) is non-constant and depends only on the variables from the \(j\)-th factor.

If no such representation exists, then \(F\) is called \emph{non-Cartesian}.
\end{definition}

\begin{remark}
When \(k=2\) and \(n_1=n_2=2\), this definition recovers the earlier notion. Indeed, writing
\[
x_1=(x,y),\qquad x_2=(z,t),
\]
the condition becomes
\[
F(x,y,z,t)=G_1(x,y)H_1(x,y,z,t)+G_2(z,t)H_2(x,y,z,t),
\]
with \(G_1\) and \(G_2\) non-constant.
\end{remark}

\begin{remark} As in Remark 1.2, one may, without loss of generality, assume that each \(G_j\) is irreducible.
\end{remark}

\begin{theorem}[Higher-dimensional genericity]
\label{thm:generic-higher}
Let $n_1 \ge n_2 \ge \cdots \ge n_k \ge 2$. 
In a broad range of dimensions (in particular when $n_1 \ge k$ or when $d$ is sufficiently large), a Zariski generic polynomial on 
\[
\mathbb{C}^{n_1} \times \cdots \times \mathbb{C}^{n_k}
\]
of degree $d \ge 2$ is non-Cartesian.
\end{theorem}

These results show that Cartesian polynomials form a thin algebraic subset, and hence represent a highly constrained and exceptional class.

\medskip

\noindent
\textbf{Algebraic and algorithmic characterization.}  Theorem \ref{thm:generic-higher} also provides an efficient sufficient criterion for a  polynomial to be non-Cartesian. Our next result shows that a polynomial is Cartesian if and only if its coefficients satisfy certain explicit systems of polynomial equations.

\begin{theorem}[Algorithmic detection]
\label{thm:algorithm}
Let $F \in \mathbb{C}[x,y,z,t]$ be a polynomial of degree $d$. 
Then the question of whether $F$ is Cartesian can be reduced to a finite system of polynomial equations of bounded degree, and hence can be decided using Gr\"obner basis methods.
\end{theorem}

This reduction relies on quantitative bounds derived from effective versions of Hilbert's Nullstellensatz, which allow us to control the degrees of the auxiliary polynomials appearing in a Cartesian decomposition.

\medskip

\noindent
\textbf{Structural results and special cases.}
We also investigate the structure of Cartesian polynomials in specific families. 
In particular, for quadratic-type polynomials of the form
\[
F(x,y,z,t) = (x - z)^2 + (y - t)^2 + R(z,t),
\]
(which appear in the proof of Proposition \ref{Proposition1}) we obtain a complete characterization of when $F$ is Cartesian in terms of algebraic divisibility properties of $R(z,t)$. 
This provides a concrete and explicit description of Cartesian structure in a natural geometric setting.

\medskip

\noindent
\textbf{Applications to incidence geometry.}
Finally, we connect the algebraic theory developed in this paper to incidence geometry.
We show that non-Cartesian polynomials satisfy sharp intersection bounds, and we construct explicit configurations demonstrating the optimality of these estimates.
In particular, we recover and extend the sharp $n^{4/3}$ bounds for non-Cartesian polynomials on $\mathbb C^2\times \mathbb C^2$ obtained in \cite{MojarradPhamValculescuZeeuw}, together with higher-dimensional analogues exhibiting different growth regimes.

The higher-dimensional setting introduces new subtleties \cite{SharirSolomon}. In particular, naive extensions of the two-block theory fail in higher dimensions, as shown by explicit counterexamples. To obtain meaningful higher-dimensional incidence bounds, one must impose stronger fibrewise non-Cartesian conditions, which form the basis of the results proved later in the paper.

Taken together, our results provide a unified perspective on Cartesian structure, revealing it as a rare algebraic phenomenon with strong implications for combinatorial geometry.

\medskip
\noindent
\textbf{Overview of the paper.}
In Section~2 we design an algorithm to detect whether a polynomial is (non)-Cartesian. 
Section~3 develops  effective sufficient criteria for being non-Cartesian on $\mathbb{C}^2\times \mathbb{C}^2$, and establishes the genericity of  non-Catersian polynomials in this setting.  
In Section~4 we extend the results in Section 3 and intersection estimates to higher dimensions.  
Finally, in Section~5 we present (counter)-examples concerning the sharpness of our results.

\medskip

\section{Algorithms for Cartesian polynomials on $\mathbb{C}^2\times \mathbb{C}^2$}

In this section we present algorithmic approaches addressing the following two questions.

\medskip

\noindent
\textbf{Question 1.} Is there an algorithm to determine whether a given polynomial 
\[
F(x,y,z,t) \in \mathbb{C}[x,y,z,t]
\]
is Cartesian?

\medskip

\noindent
\textbf{Question 2.} Let $\mathcal{C}_n$ denote the set of Cartesian polynomials of degree at most $n$. 
\begin{itemize}
\item Is $\mathcal{C}_n$ closed under the usual (Euclidean) topology on $\mathbb{C}[x,y,z,t]$?
\item Is there an algorithm to compute its closure?
\end{itemize}

While here we present the results only for the case of a single polynomial (corresponding to hypersurfaces), all the results can be extended easily to the case of subvarieties of higher codimensions. Similarly, the result can be extended to polynomials in higher dimensions $\mathbb{C}^{n_1}\times \ldots \times \mathbb{C}^{n_k}$. 

\subsection{Some preliminaries}

We begin with a concrete example to illustrate the main features of the problem and the algebraic structure underlying the algorithms.

\medskip

Assume that in Question 1 we are given the polynomial
\[
F(x,y,z,t)
= x^2 + y^2 + z^2 + t^2 + xy + 2xz + 3xt + 4yz + 5yt + 6zt + 1 - 10x,
\]
and we seek a representation
\[
F = G(x,y)\,F_1(x,y,z,t) + H(z,t)\,F_2(x,y,z,t),
\]
under the additional constraint that the polynomials $G,H,F_1,F_2$ all have degree at most $2$.

\medskip

\paragraph{\bf Parametrization of unknown polynomials.}

We write $G$ and $H$ in terms of their coefficients:
\begin{align*}
G(x,y)
&= a_{G,0} + a_{G,1}x + a_{G,2}y + a_{G,11}x^2 + a_{G,12}xy + a_{G,22}y^2,\\
H(z,t)
&= a_{H,0} + a_{H,1}z + a_{H,2}t + a_{H,11}z^2 + a_{H,12}zt + a_{H,22}t^2,
\end{align*}
where all coefficients lie in $\mathbb{C}$.

Similarly, we parametrize $F_1$ and $F_2$ by their coefficients. Since these are polynomials in four variables, they involve a larger collection of coefficients. For instance,
\begin{align*}
F_1(x,y,z,t)
&= a_{F_1,0} + a_{F_1,1}x + a_{F_1,2}y + a_{F_1,11}x^2 + a_{F_1,12}xy + a_{F_1,22}y^2 \\
&\quad + a_{F_1,3}z + a_{F_1,4}t + a_{F_1,33}z^2 + a_{F_1,34}zt + a_{F_1,44}t^2 + \cdots,\\
F_2(x,y,z,t)
&= a_{F_2,0} + a_{F_2,1}x + a_{F_2,2}y + a_{F_2,11}x^2 + a_{F_2,12}xy + a_{F_2,22}y^2 \\
&\quad + a_{F_2,3}z + a_{F_2,4}t + a_{F_2,33}z^2 + a_{F_2,34}zt + a_{F_2,44}t^2 + \cdots .
\end{align*}

\paragraph{\bf Reduction to a system of algebraic equations.}

Substituting these expressions into the identity
\[
F = G F_1 + H F_2,
\]
and expanding the right-hand side, we collect coefficients of all monomials in $x,y,z,t$ and equate them with those of $F$.

This produces a system of polynomial equations in the unknown coefficients
\[
a_{G,i}, \quad a_{H,j}, \quad a_{F_1,k}, \quad a_{F_2,\ell}.
\]

For example:
\begin{itemize}
\item Comparing coefficients of $x$ gives
\[
-10 = a_{G,0} a_{F_1,1} + a_{G,1} a_{F_1,0}.
\]
\item Comparing coefficients of $xz$ gives
\[
2 = a_{G,0} a_{F_1,13} + a_{G,1} a_{F_1,3}
    + a_{H,0} a_{F_2,13} + a_{H,3} a_{F_2,1}.
\]
\end{itemize}

Each such equation has degree at most $2$.

\paragraph{\bf Non-constant constraints.}

We must also impose that $G$ and $H$ are non-constant. This means:
\[
\text{at least one } a_{G,i} \ (i \neq 0) \text{ is nonzero, and similarly for } H.
\]

This condition can be encoded algebraically using the standard trick:
a complex number $a$ is nonzero if and only if there exists $t \in \mathbb{C}$ such that
\[
1 - a t = 0.
\]

Thus, the non-constancy conditions can be incorporated into the system by introducing auxiliary variables.

\paragraph{\bf Conclusion for Question 1 (bounded-degree case).}

We conclude that, in this bounded-degree setting, determining whether such a representation exists reduces to deciding whether a finite system of polynomial equations has a solution over $\mathbb{C}$.

This problem can be solved algorithmically using Gr\"obner bases \cite{Barakat-Hegermann}.

\medskip

\paragraph{\bf Extension to Question 2 (bounded-degree case).}

We now consider the family of all polynomials $F$ of degree at most $2$ that admit a decomposition
\[
F = G(x,y)F_1 + H(z,t)F_2,
\]
with $G,H,F_1,F_2$ also of degree at most $2$.

In this case, the coefficients of $F$ are also treated as variables, say $\{a_{F,h}\}$, and we obtain a larger system of polynomial equations involving:
\[
a_{G,i}, \quad a_{H,j}, \quad a_{F_1,k}, \quad a_{F_2,\ell}, \quad a_{F,h}.
\]

To understand the structure of this set (for example, whether it is large, or whether it satisfies nontrivial constraints), we eliminate the variables corresponding to $G,H,F_1,F_2$, retaining only the variables $\{a_{F,h}\}$.

This elimination problem can again be handled using Gr\"obner basis methods.

\medskip

In particular, this allows us to determine whether there exist algebraic relations among the coefficients of $F$, and hence to study the size and closure properties of the set $\mathcal{C}_2$.

\subsection{Algorithms}

The example in the previous subsection illustrates that, if one can bound the degrees of $G,H,F_1,F_2$ in terms of the degree of $F$, then Questions~1 and~2 admit an algorithmic solution. In this subsection, we establish the bounds needed for this purpose.

We recall that an ideal $\mathcal{I}\subset \mathbb{C}[x,y,z,t]$ is said to be \emph{radical} if, whenever $f^N \in \mathcal{I}$ for some polynomial $f$ and some positive integer $N$, it follows that $f \in \mathcal{I}$. The following lemma also appears in \cite{MojarradPhamValculescuZeeuw}. Here we present it in the simpler form needed for later applications, together with a streamlined proof.

\begin{lemma}
Let $G(x,y)$ and $H(z,t)$ be irreducible, non-constant polynomials, and let $\mathcal{I}\subset \mathbb{C}[x,y,z,t]$ be the ideal generated by $G(x,y)$ and $H(z,t)$. Then $\mathcal{I}$ is radical.
\label{RadicalIdealLemma}
\end{lemma}

\begin{proof}
Let $f\in \mathbb{C}[x,y,z,t]$ and assume that $f^N\in \mathcal{I}=(G,H)$ for some positive integer $N$. Then there exist polynomials $f_1,f_2\in \mathbb{C}[x,y,z,t]$ such that
\[
f^N = G f_1 + H f_2 .
\]

If every zero of $H$ is also a zero of $f^N$, then since $H$ is irreducible it follows that $H \mid f^N$. Because $\mathbb{C}[x,y,z,t]$ is a UFD, this implies $H\mid f$, and hence $f\in (H)\subset (G,H)=\mathcal{I}$. Therefore, we may assume that there exists $(z_0,t_0)$ such that
\[
H(z_0,t_0)=0, \qquad\text{but}\qquad f(x,y,z_0,t_0)\not\equiv 0.
\]

Fix a monomial order in which $x,y \gg z,t$. Then we may write
\[
f(x,y,z,t)=G(x,y)\,p(x,y,z,t)+R(x,y,z,t),
\]
where $R(x,y,z,t)$ is the remainder of $f$ upon division by $G(x,y)$ (see, e.g., standard references on Gr\"obner bases). By definition, no monomial in $R$ is divisible by the leading monomial of $G(x,y)$.

Since $G(x,y)$ depends only on $x,y$, it follows that for any fixed $(z_0,t_0)$, if $R(x,y,z_0,t_0)$ is not identically zero, then it still contains no monomial divisible by the leading monomial of $G(x,y)$. Hence $R(x,y,z_0,t_0)$ is already reduced with respect to $G(x,y)$.

Now let $(z_0,t_0)$ be a zero of $H(z,t)$. Then
\[
f(x,y,z_0,t_0)\in (G(x,y)).
\]
Since $G(x,y)$ is irreducible, it forms a Gr\"obner basis for the ideal it generates. By the characterization of Gr\"obner bases, the remainder of $f(x,y,z_0,t_0)$ upon division by $G(x,y)$ must be zero. But this remainder is precisely $R(x,y,z_0,t_0)$. Therefore,
\[
R(x,y,z_0,t_0)\equiv 0.
\]

From the previous paragraph, we conclude that $R(x,y,z_0,t_0)$ vanishes identically for every zero $(z_0,t_0)$ of $H(z,t)$. Hence the polynomial
\[
f(x,y,z,t)-G(x,y)p(x,y,z,t)=R(x,y,z,t)
\]
vanishes whenever $H(z,t)=0$. Since $H(z,t)$ is irreducible, it follows that $R(x,y,z,t)$ is divisible by $H(z,t)$. That is, there exists a polynomial $q(x,y,z,t)$ such that
\[
f(x,y,z,t)=G(x,y)p(x,y,z,t)+H(z,t)q(x,y,z,t).
\]
Therefore $f\in (G,H)=\mathcal{I}$, completing the proof.
\end{proof}

The explicit bounds below are required for the algorithm that determines whether a polynomial is Cartesian. ({\bf Remark:} In \cite{MojarradPhamValculescuZeeuw}, at the end of Section 2.3, the authors asserted that for the decomposition $f=Gp+Hq$ as in the proof of the above lemma, we have $\deg (f)\geq \deg (Gp), \deg (Hq)$. However, there is no complete argument to support this claim, and it is unclear if the claim is correct. A reason is that choosing a monomial order which preserves total degrees, as used by the authors in the assertion, does not allow one to show that $f$ can be written as $Gp+Hq$.)

\begin{lemma}\label{DegreeBoundLemma}

\textbf{A)} Let $f(x,y,z,t)$ be a Cartesian polynomial of degree $d$. Then there exist polynomials $G(x,y)$, $H(z,t)$ and $q(x,y,z,t,u)$, $q_1(x,y,z,t,u)$, $q_2(x,y,z,t,u)$ such that:

\begin{enumerate}
\item[(i)] The identity
\[
q(1 - u f) + q_1 G + q_2 H \equiv 1
\]
holds in $\mathbb{C}[x,y,z,t,u]$.

\item[(ii)] The degrees of $G$ and $H$ satisfy
\[
1 \leq \deg G, \deg H \leq d.
\]

\item[(iii)] The degrees of the products satisfy
\[
\deg\big(q(1 - u f)\big),\ \deg(q_1 G),\ \deg(q_2 H)
\;\leq\; \max\{(d+1)d^2, 27\}.
\]
Moreover, when $d=2$, this bound can be improved to $24$.
\end{enumerate}

\medskip

\noindent
\textbf{B)} Conversely, suppose that there exist polynomials $f$, $G$, $H$, $q$, $q_1$, $q_2$, where $G(x,y)$ and $H(z,t)$ are irreducible non-constant polynomials, satisfying conditions \textnormal{(i)}--\textnormal{(iii)} above. Then $f$ is Cartesian.

\end{lemma}

\begin{proof}

\textbf{A)} Since $f$ is Cartesian, there exist irreducible non-constant polynomials $G(x,y)$ and $H(z,t)$ such that
\[
f \in (G,H).
\]
By Lemma~\ref{RadicalIdealLemma}, we may choose $G$ and $H$ so that
\[
1 \leq \deg G, \deg H \leq d,
\]
which establishes condition \textnormal{(ii)}.

\medskip

To obtain condition \textnormal{(i)}, we introduce a new variable $u$ and consider the polynomials $1 - u f$, $G$, and $H$. Since $f \in (G,H)$, these polynomials generate the unit ideal in $\mathbb{C}[x,y,z,t,u]$. Therefore, by the effective Hilbert Nullstellensatz, there exist polynomials $q, q_1, q_2 \in \mathbb{C}[x,y,z,t,u]$ such that
\[
q(1 - u f) + q_1 G + q_2 H = 1.
\]

\medskip

Furthermore, degree bounds in the effective Nullstellensatz (see \cite{Brownawell} \cite{Kollar}\cite{Sombra}) imply that, if
\[
\deg f, \deg G, \deg H \leq d,
\]
then one may choose $q, q_1, q_2$ so that
\[
\deg\big(q(1 - u f)\big),\ \deg(q_1 G),\ \deg(q_2 H)
\;\leq\; \max\{(d+1)d^2, 27\}.
\]
When $d=2$, this bound improves to $24$. This establishes condition \textnormal{(iii)}.

\medskip

This completes the proof of part \textbf{A}.

\medskip

\textbf{B)} Suppose that condition \textnormal{(i)} holds. Let $(x,y,z,t)$ be a point such that $f(x,y,z,t)=0$. Then
\[
q(1 - u f) + q_1 G + q_2 H = 1
\quad \Longrightarrow \quad
q + q_1 G + q_2 H = 1,
\]
and hence $G(x,y)$ and $H(z,t)$ cannot both vanish at this point. Therefore,
\[
\{f = 0\} \subset \{G = 0\} \cup \{H = 0\}.
\]

Equivalently, every common zero of $G$ and $H$ is a zero of $f$, and hence
\[
f \in \sqrt{(G,H)}.
\]

\medskip

Since $G$ and $H$ are irreducible and non-constant, Lemma~\ref{RadicalIdealLemma} implies that the ideal $(G,H)$ is radical. Therefore,
\[
f \in (G,H),
\]
and hence $f$ is Cartesian.

\end{proof}

\subsection{The closure of the set of Cartesian polynomials} 

Fix an integer $d\geq 3$. Here we first demonstrate that the set of Cartesian polynomials of degree at most $d$ is not closed (in the Zariski topology). This is the same as that the set of Cartesian polynomials is not closed in the usual topology on the same set, regarded as an affine space. To this end, let $F(x,y,z,t)$ be a {\bf non-Cartesian} polynomial of degree at most $d-1\geq 2$ (whose existence is discussed in the following sections). For each $n$, the polynomial $F_n(x,y,z,t)=(1+\frac{x}{n})F(x,y,z,t)$ is a Cartesian polynomial of degree at most $d$, and the sequence $\{F_n\}$ converges to the polynomial $F(x,y,z,t)$ which is non-Cartesian.  

Since the set of Cartesian polynomials of degree $\leq d$ is a union of images of a finite set of  polynomial maps between affine spaces  (by results in the previous Subsections), both it and its closure can be computed by Gr\"obner basis techniques, see \cite{Barakat-Hegermann}.

\section{Genericity of non-Cartesian polynomials on $\mathbb{C}^2\times \mathbb{C}^2$}

In this section, we present several results on constructing non-Cartesian polynomials. These results also allow us to show that a generic polynomial is non-Cartesian.

First, we provide an efficient sufficient criterion for a polynomial to be non-Cartesian. 

\begin{theorem}\label{TheoremSufficientConditions}
Assume that $F(x,y,z,t)$ is a non-zero polynomial satisfying the following conditions:
\begin{enumerate}
\item $F$ is irreducible and $F \notin \mathbb{C}[z,t]$.

\item The set
\[
\{(z_0,t_0)\in \mathbb{C}^2 : F(x,y,z_0,t_0)\ \text{is reducible in } \mathbb{C}[x,y]\}
\]
is finite.

\item For each $(z_0,t_0)\in \mathbb{C}^2$, the set
\[
\left\{(z_1,t_1)\in \mathbb{C}^2 :
\exists \lambda \in \mathbb{C} \text{ such that }
F(x,y,z_0,t_0) = \lambda F(x,y,z_1,t_1)
\right\}
\]
is finite.
\end{enumerate}
Then $F$ is non-Cartesian.
\end{theorem}

\begin{proof}
Assume, for contradiction, that $F$ is Cartesian. Then there exist non-constant polynomials $G(x,y)$ and $H(z,t)$, together with polynomials $F_1(x,y,z,t)$ and $F_2(x,y,z,t)$, such that
\[
F = G F_1 + H F_2.
\]
Without loss of generality, we may assume that both $G$ and $H$ are irreducible.

\medskip

Since $F$ is irreducible and $F \notin \mathbb{C}[z,t]$, while $H(z,t)$ is non-constant, it follows that $F$ does not divide $H$. Consequently, for a generic zero $(z_0,t_0)$ of $H(z,t)$ (note that $H$ has infinitely many zeros), the specialization $F(x,y,z_0,t_0)$ is not identically zero.

\medskip

Substituting such a point $(z_0,t_0)$ into the identity $F = G F_1 + H F_2$, we obtain
\[
F(x,y,z_0,t_0) = G(x,y)\,F_1(x,y,z_0,t_0).
\]
Thus, $F(x,y,z_0,t_0)$ admits a nontrivial factorization unless $F_1(x,y,z_0,t_0)$ is a constant.

\medskip

By condition (2), the set of $(z_0,t_0)$ for which $F(x,y,z_0,t_0)$ is reducible is finite. Hence, for all but finitely many zeros $(z_0,t_0)$ of $H$, the polynomial $F(x,y,z_0,t_0)$ is irreducible. Since $G(x,y)$ is non-constant, this forces
\[
F_1(x,y,z_0,t_0) \in \mathbb{C} \setminus \{0\}.
\]

\medskip

Now choose two distinct generic zeros $(z_0,t_0)$ and $(z_1,t_1)$ of $H$ satisfying the above property. Then
\[
F(x,y,z_0,t_0) = G(x,y)\,F_1(x,y,z_0,t_0), 
\qquad
F(x,y,z_1,t_1) = G(x,y)\,F_1(x,y,z_1,t_1),
\]
where both $F_1(x,y,z_0,t_0)$ and $F_1(x,y,z_1,t_1)$ are non-zero constants.

\medskip

It follows that there exists $\lambda \in \mathbb{C}$ such that
\[
F(x,y,z_0,t_0) = \lambda F(x,y,z_1,t_1),
\]
which contradicts condition (3) and completes the proof.
\end{proof}

\begin{theorem}
For a generic polynomial $F(x,y,z,t)$ of degree $d\geq 3$, the following properties are satisfied: 

1. $F(x,y,z,t)$ is irreducible and does not belong to $\mathbb{C}[z,t]$. 

2. There are at most finitely many $(z_0,t_0)$ such that $F(x,y,z_0,t_0)$ is reducible. 

3. If $F(x,y,z_0,t_0)=\lambda F(x,y,z_1,t_1)$ as polynomials in $\mathbb{C}[x,y]$, for some complex constants $z_1,t_1,\lambda \in \mathbb{C}$, then $\lambda =1$ and $(z_1,t_1)=(z_0,t_0)$. 

\label{TheoremGenericPolynomials}
\end{theorem}

\begin{proof}

We first recall the following fact: the space $\mathcal{P}_{n;d}$ of polynomials in $n$ variables of degree at most $d$ is a vector space of dimension $\binom{d+n}{n}$. 

On the other hand, the set $\mathcal{P}_{n;d_1,d_2}$ of polynomials in $n$ variables which are a product of two polynomials of degrees $d_1,d_2\geq 1$ has dimension 
\[
\binom{d_1+n}{n}+\binom{d_2+n}{n}-1.
\]
Here we subtract $1$ from the sum $\binom{d_1+n}{n}+\binom{d_2+n}{n}$ because in a product $P_1P_2$, we can normalise so that the factor $P_2$ has leading coefficient $1$.  

1. The set of reducible polynomials in $\mathcal{P}_{n;d}$ is therefore a Zariski constructible subset (this is similar to the way we showed above that the set of Cartesian polynomials with degree at most $d$ is constructible), whose dimension is at most 
\[
\max_{d_1,d_2\geq 1,\ d_1+d_2=d}\left(\binom{d_1+n}{n}+\binom{d_2+n}{n}-1\right).
\]

Note that each of the terms $\binom{d_1+n}{n}+\binom{d_2+n}{n}$ is $\leq \binom{d+n}{n}$ (see below for more precise estimates), and hence $\mathcal{P}_{n;d_1,d_2}$ has codimension at least $1$ in $\mathcal{P}_{n;d}$. Therefore, a generic polynomial in $\mathcal{P}_{n;d}$ is irreducible. Applying this for $n=4$, the proof of (1) is complete.  

2. In this case, we apply the above arguments for $n=2$. We first show that when $d\geq 3$, for integers $d_1,d_2\geq 1$ with $d_1+d_2=d$, we have
\[
\binom{d_1+2}{2}+\binom{d_2+2}{2}-1 \leq \binom{d+2}{2}-2,
\]
i.e. the Zariski constructible set $\mathcal{P}_{2;d_1,d_2}$ has codimension at least $2$ in $\mathcal{P}_{2;d}$. 

In fact, this amounts to 
\[
\frac{d_1^2+3d_1+2}{2}+\frac{d_2^2+3d_2+2}{2}-1 \leq \frac{d^2+3d+2}{2}-2.
\]
Using $d=d_1+d_2$ and 
\[
d^2=(d_1+d_2)^2=(d_1^2+d_2^2)+2d_1d_2 \geq (d_1^2+d_2^2)+4,
\]
the inequality reduces to
\[
d_1d_2-2\geq 0,
\]
which holds under the assumptions $d_1,d_2\geq 1$ and $d_1+d_2=d\geq 3$. 

Now we consider the evaluation map
\[
\varphi : \mathcal{P}_{4;d}\times \mathbb{C}^2 \to \mathcal{P}_{2;d},
\]
given by $\varphi(F,(z,t))=$ the polynomial $(x,y)\mapsto F(x,y,z,t)$. Let 
\[
\mathcal{I}=\varphi^{-1}(0)\subset \mathcal{P}_{4;d}\times \mathbb{C}^2
\]
(here $0\in \mathcal{P}_{2;d}$ is the constant zero polynomial). 

Note that $\mathcal{I}$ is an algebraic subvariety of $\mathcal{P}_{4;d}\times \mathbb{C}^2$, and any of its irreducible components has dimension at least the relative dimension of $\varphi$ (which is, by definition, the difference between the dimension of the source $\mathcal{P}_{4;d}\times \mathbb{C}^2$ and the dimension of the target $\mathcal{P}_{2;d}$). 

For any $h\in \mathcal{P}_{2;d}$, we have 
\[
\varphi^{-1}(h)=(h,(0,0))+\mathcal{I},
\]
which is the translation of $\mathcal{I}$. Indeed, $(F,(z,t))\in \varphi^{-1}(h)$ if and only if $(F-h,(z,t))\in \varphi^{-1}(0)$. 

It follows that $\varphi$ is surjective and equidimensional: every fiber has dimension equal to the relative dimension of $\varphi$. 

[Indeed, from the description of $\varphi^{-1}(h)$ above, all fibers are isomorphic to $\mathcal{I}$. Let $m$ be the dimension of $\mathcal{I}$ (i.e. the maximum dimension of its irreducible components). Then 
\[
m+\dim(\mathcal{P}_{2;d})=\dim(\mathcal{P}_{4;d}\times \mathbb{C}^2),
\]
so $m$ equals the relative dimension of $\varphi$. Moreover, every irreducible component of $\mathcal{I}$ has dimension $\geq$ the relative dimension, by the Fiber Dimension Theorem.]

Therefore, if $\mathcal{R}_{2;d}\subset \mathcal{P}_{2;d}$ is the set of reducible polynomials, then 
\[
\varphi^{-1}(\mathcal{R}_{2;d})\subset \mathcal{P}_{4;d}\times \mathbb{C}^2
\]
is Zariski constructible of codimension equal to that of $\mathcal{R}_{2;d}$ in $\mathcal{P}_{2;d}$, and hence $\geq 2$.

Now we complete the proof that for a generic polynomial $F\in \mathcal{P}_{4;d}$, the set
\[
W(F)=\{(z_0,t_0)\in \mathbb{C}^2: F(x,y,z_0,t_0)\ \text{is reducible}\}
\]
is finite. 

Let $\pi_1:\mathcal{P}_{4;d}\times \mathbb{C}^2 \to \mathcal{P}_{4;d}$ be the natural projection. Then $\pi_1$ is surjective, and all its fibers have dimension $2$. It is easy to check that
\[
W(F)=\pi_1^{-1}(\{F\})\cap \varphi^{-1}(\mathcal{R}_{2;d}).
\]

Hence, for a generic $F$, $W(F)$ has codimension $\geq 2$ in $\mathbb{C}^2$, and therefore is finite. 

[Alternatively, if for a generic $F$ the set $\pi_1^{-1}(\{F\})\cap \varphi^{-1}(\mathcal{R}_{2;d})$ were not finite, then its dimension would be at least $1$. This would imply that $\varphi^{-1}(\mathcal{R}_{2;d})$ has dimension at least $1+\dim(\mathcal{P}_{4;d})$, hence codimension $\leq 1$, contradicting the previously established codimension $\geq 2$.]

3. Assume that 
\[
F(x,y,z,t)=a_0x^d+b_0y^d+(c_{1,1}z+c_{1,2}t)x^{d-1}+(c_{2,1}z+c_{2,2}t)y^{d-1}+\ldots
\]
with $a_0,b_0\neq 0$, and 
\[
A=\begin{pmatrix} c_{1,1} & c_{1,2} \\ c_{2,1} & c_{2,2} \end{pmatrix}
\]
invertible. The set of such polynomials is Zariski open dense in $\mathcal{P}_d$. 

Now assume that 
\[
F(x,y,z_0,t_0)=\lambda F(x,y,z_1,t_1)
\]
in $\mathbb{C}[x,y]$, for some $\lambda,z_0,t_0,z_1,t_1\in \mathbb{C}$. By comparing the coefficient of $x^d$, we obtain $\lambda=1$. Then, comparing the coefficients of $x^{d-1}$ and $y^{d-1}$, we obtain 
\[
A\cdot (z_0,t_0)=A\cdot (z_1,t_1).
\]
Since $A$ is invertible, this implies $(z_0,t_0)=(z_1,t_1)$. 

\end{proof}

In the above theorem, note that the proofs of Properties 1 and 3 are also valid for $d=2$, while Property 2 is not satisfied for $d=2$ by the following lemma.

\begin{lemma}
Let $F(x,y,z,t)$ be a generic polynomial of degree $2$. Then there are infinitely many $(z_0,t_0)$ such that $F(x,y,z_0,t_0)$ is reducible. 
\label{LemmaDegree2Polynomials}
\end{lemma}

\begin{proof}

Since $F\in\mathbb{C}[x,y,z,t]$ is of degree $2$, it can be written as
\begin{equation}\label{eq:quad_decomp}
F(x,y,z,t)
=
Q(x,y)
+ z\,L_1(x,y)
+ t\,L_2(x,y)
+ R(z,t),
\end{equation}
where $Q(x,y)$ is quadratic in $(x,y)$, $L_1,L_2$ are linear forms in $(x,y)$,
and $R(z,t)$ is a quadratic polynomial in $(z,t)$.

Now fix $(z,t)\in\mathbb{C}^2$. The fiber
\[
F_{z,t}(x,y):=F(x,y,z,t)
\]
is a quadratic polynomial in $\mathbb{C}[x,y]$, hence can be written as
\begin{equation}\label{eq:fiber_coeffs}
F_{z,t}(x,y)
=
a(z,t)x^2+b(z,t)xy+c(z,t)y^2
+d(z,t)x+e(z,t)y+f(z,t),
\end{equation}
where each coefficient $a,b,c,d,e,f$ depends polynomially on $(z,t)$.

A quadratic polynomial in $\mathbb{C}[x,y]$ is reducible if and only if its
homogenization
\[
\widetilde F(x,y,u)
=
ax^2+bxy+cy^2+dxu+eyu+fu^2
\]
factors into linear forms in $\mathbb{C}[x,y,u]$.
Equivalently, if
\[
\widetilde F(x,y,u)
=
(x,y,u)\,M\,(x,y,u)^{\mathsf T},
\qquad
M=
\begin{pmatrix}
a & b/2 & d/2\\
b/2 & c & e/2\\
d/2 & e/2 & f
\end{pmatrix},
\]
then $\widetilde F$ is reducible if and only if $\det M=0$.
Thus $F_{z,t}$ is reducible in $\mathbb{C}[x,y]$ if and only if
\begin{equation}\label{eq:discriminant}
\Delta(z,t)
:=
\det
\begin{pmatrix}
a(z,t) & b(z,t)/2 & d(z,t)/2\\
b(z,t)/2 & c(z,t) & e(z,t)/2\\
d(z,t)/2 & e(z,t)/2 & f(z,t)
\end{pmatrix}
=0.
\end{equation}

The quantity $\Delta(z,t)$ is a polynomial function of $(z,t)$. If it is not a non-zero constant, then there are infinitely many $(z,t)$ for which $\Delta(z,t)=0$, and hence $F_{z,t}(x,y)$ is reducible. 

Now, the condition $\Delta(z,t)\equiv 0$ (or even $\Delta(z,t)\equiv \text{constant}$)
imposes polynomial equations on the coefficients of $F$, and hence defines a
Zariski closed subset of the space of quadratic polynomials.
This subset is proper. For example,
\[
F(x,y,z,t)=x^2+y^2+zx+ty
\]
yields $\Delta(z,t)=-(z^2+t^2)$, which is not constant.

Therefore, for generic $F$, the polynomial $\Delta(z,t)$ is non-constant. It then follows from the above arguments that for a generic $F$, there are infinitely many $(z,t)$ for which $F_{z,t}(x,y)$ is reducible. 

\end{proof}

Instead of the above negative result, we can still obtain the following.

\begin{proposition}\label{Proposition1}
A generic polynomial $F(x,y,z,t)$ of degree $d \geq 2$ is non-Cartesian. Moreover, for any polynomial $G$, the set
\[
\{a \in \mathbb{C} : F + aG \text{ is Cartesian}\}
\]
is finite.
\end{proposition}

\begin{proof}
For the case $d \geq 3$, the statement follows directly from Theorems \ref{TheoremSufficientConditions} and \ref{TheoremGenericPolynomials}.

\medskip

For $d=2$, Lemma \ref{LemmaDegree2Polynomials} shows that the above argument does not apply. We therefore present a separate argument for this special case.

\medskip

We first reduce to a normal form. It suffices to consider polynomials of the form
\[
F(x,y,z,t) = (x - z)^2 + (y - t)^2 + R(z,t),
\]
where $R(z,t)$ is a non-zero polynomial of degree at most $2$.

\medskip

Indeed, a general quadratic polynomial can be written as
\[
F(x,y,z,t)
= a_{11}x^2 + a_{12}xy + a_{22}y^2
+ a_{13}xz + a_{14}xt + a_{23}yz + a_{24}yt
+ R(z,t),
\]
where $R(z,t)$ has degree at most $2$.

\medskip

We note that the property of being Cartesian is preserved under polynomial automorphisms of $\mathbb{C}^2_{x,y}$ and $\mathbb{C}^2_{z,t}$.

\medskip

First, applying a linear automorphism of $\mathbb{C}^2_{x,y}$ (using the generic assumption on $F$), we may diagonalize the quadratic part in $(x,y)$ and reduce $F$ to the form
\[
x^2 + y^2 + a_{13}xz + a_{14}xt + a_{23}yz + a_{24}yt + R(z,t).
\]

\medskip

Next, applying a linear automorphism of $\mathbb{C}^2_{z,t}$, namely
\[
Z = -\tfrac{1}{2}(a_{13}z + a_{14}t), \qquad
T = -\tfrac{1}{2}(a_{23}z + a_{24}t),
\]
we reduce $F$ to the form
\[
(x - Z)^2 + (y - T)^2 + R(Z,T),
\]
as claimed.

\medskip

Hence, from now on, we assume that
\[
F(x,y,z,t) = (x - z)^2 + (y - t)^2 + R(z,t),
\]
where $R(z,t)$ has degree at most $2$ and is non-zero.

\medskip

We now prove the following stronger statement.

\medskip

\noindent\textbf{Claim.}
$F$ is Cartesian if and only if there exists a constant $c \in \mathbb{C}$ such that $R(z,t)$ is divisible by $z + it - c$ or $z - it - c$.

\medskip

We prove the claim in several steps.

\medskip

\noindent\textbf{Step 1.}
Let
\[
Q(x,y) = (x - p)^2 + (y - q)^2 + r,
\]
where $p,q,r \in \mathbb{C}$. Then $Q(x,y)$ is reducible if and only if $r = 0$.

\medskip

Indeed (see also Lemma \ref{LemmaDegree2Polynomials}), $Q(x,y)$ is irreducible if and only if its homogenization
\[
Q(x,y,u) = (x - pu)^2 + (y - qu)^2 + r u^2
\]
is non-degenerate. Since non-degeneracy is invariant under linear changes of variables, applying the change
\[
X = x - pu, \quad Y = y - qu,
\]
we obtain
\[
Q(X,Y,u) = X^2 + Y^2 + r u^2,
\]
which is non-degenerate if and only if $r \neq 0$. This proves the claim.

\medskip

\noindent\textbf{Step 2.}
Suppose that $F$ is Cartesian, so that
\[
F = G(x,y)F_1(x,y,z,t) + H(z,t)F_2(x,y,z,t),
\]
where $G$ and $H$ are non-constant irreducible polynomials.

\medskip

We claim that $R(z,t)=0$ whenever $H(z,t)=0$.

\medskip

Indeed, if $R(z,t)$ does not vanish identically on the set $\{H=0\}$, then there are only finitely many points $(z_0,t_0)$ such that
\[
H(z_0,t_0) = R(z_0,t_0) = 0.
\]
By Step 1, there are then only finitely many $(z_0,t_0)$ such that $H(z_0,t_0)=0$ and
\[
Q_{z_0,t_0}(x,y) = (x - z_0)^2 + (y - t_0)^2 + R(z_0,t_0)
\]
is reducible.

\medskip

Since Property (1) holds whenever $R \neq 0$ (by choosing $(z_0,t_0)$ with $R(z_0,t_0)\neq 0$), and Property (2) always holds (see the proof of Theorem \ref{TheoremSufficientConditions}), we may apply Theorem \ref{TheoremSufficientConditions} to conclude that $F$ is non-Cartesian, a contradiction. This proves the claim.

\medskip

\noindent\textbf{Step 3.}
Assume $H(z_0,t_0)=0$. Then $R(z_0,t_0)=0$, and hence
\[
(x - z_0)^2 + (y - t_0)^2
= (x + iy - (z_0 + it_0))(x - iy - (z_0 - it_0))
\]
divides $G(x,y)$.

\medskip

Since $G$ is independent of $(z,t)$ and $H$ has infinitely many zeros, this forces $G$ to have degree $1$. Thus, for each such $(z_0,t_0)$, $G$ must be a scalar multiple of either
\[
x + iy - (z_0 + it_0)
\quad \text{or} \quad
x - iy - (z_0 - it_0).
\]

\medskip

Without loss of generality, assume the first case. Since $G$ is fixed, we conclude that $z_0 + it_0$ must be constant on $\{H=0\}$. Thus there exists $c \in \mathbb{C}$ such that
\[
z_0 + it_0 = c
\quad \text{whenever } H(z_0,t_0)=0.
\]

\medskip

It follows that $H(z,t) = z + it - c$. Since $R(z,t)$ has degree at most $2$, is non-constant, and vanishes on $\{H=0\}$, we conclude that $R$ is divisible by $z + it - c$, i.e.
\[
R(z,t) = (z + it - c)(az + bt + d)
\]
for some constants $a,b,d \in \mathbb{C}$.

\medskip

Since the space of polynomials of degree $2$ in $(z,t)$ has dimension $6$, this condition defines a proper algebraic subset. Hence, for a generic choice of $R(z,t)$, the corresponding $F$ is non-Cartesian.

\medskip

\noindent\textbf{Step 4.}
Finally, we show that this condition is also sufficient.

\medskip

Assume that $R(z,t)$ is divisible by $z + it - c$. Consider the sets
\[
P = \{(x,y) \in \mathbb{C}^2 : x + iy - c = 0\}, \qquad
Q = \{(z,t) \in \mathbb{C}^2 : z + it - c = 0\}.
\]
Then one checks directly that $F=0$ on $P \times Q$. By the characterization of Cartesian polynomials (see \cite{MojarradPhamValculescuZeeuw}), this implies that $F$ is Cartesian.

\medskip

This completes the proof.
\end{proof}

\section{Some generalisations to higher dimensions} 

\begin{theorem} Assume that $n_1\geq n_2\geq \ldots \geq n_k\geq 2$.

1) If $n_1\geq k$ and $d\geq 2$, then a generic polynomial of degree $d$ in $\mathbb{C}^{n_1}\times \ldots \times \mathbb{C}^{n_k}$ is non-Cartesian.

2) If $d$ is large enough (it is sufficient to choose $d\geq \max \{2k-3,7\}$), then a generic polynomial of degree $d$ in $\mathbb{C}^{n_1}\times \ldots \times \mathbb{C}^{n_k}$ is non-Cartesian. 

\label{TheoremHigherDimensions}\end{theorem}

\begin{proof}

1) In fact, we know that the set of reducible polynomials in $\mathcal{P}_{d;n_1}$ has codimension at least $n_1$, while the codimension of a set $\{G_2(z_2)=\ldots =G_k(z_{k})=0\}$ is precisely $k-1$ (if $G_2,\ldots ,G_{k}$ are non-constant). Therefore, as in Theorem \ref{TheoremGenericPolynomials}, we can show that for a generic polynomial $F(x_1,\ldots ,x_k)$, and for a set of non-constant polynomials $G_2(z_2), \ldots, G_k(z_{k})$ then there are infinitely many  $(x_2,\ldots ,x_k)\in \{G_2(z_2)=\ldots =G_k(z_{k})=0\}$ such that $x_1\mapsto F(x_1,x_1,\ldots ,x_k)$ is irreducible. This corresponds to Property 2 in the statement of Theorem  \ref{TheoremGenericPolynomials}. 

The generalisation of Property 1 in the statement of  Theorem  \ref{TheoremGenericPolynomials} can be more easily proven.

Concerning the generalisation of Property 3 in the statement of Theorem \ref{TheoremGenericPolynomials}: adapting the proof of Theorem \ref{TheoremGenericPolynomials}, we need to consider $2$ cases. If $n_2+n_3+\ldots +n_k\geq n_1$, then by looking at the homogeneous part of degree $d-1$ in $x_1$, and reduce to showing that a linear map $A: \mathbb{C}^{n_2}\times \ldots \mathbb{C}^{n_k}\rightarrow \mathbb{C}^{n_1} $ has rank $n_1$. On the other hand, if $n_2+n_3+\ldots +n_k=n< n_1$, then we write $x_1=(y_1,z_1)$ where $y_1\in \mathbb{C}^n $ and $z_1\in \mathbb{C}^{n_1-n}$, and look at the homogeneous part of degree $d-1$ in $y_1$ and use the same argument like in the first case.

After having this generalisation of Theorem  \ref{TheoremGenericPolynomials}, the proof of Theorem \ref{TheoremSufficientConditions} can be also generalized to yield the conclusion.

2) Even if $n_1<k$, if $d$ is large enough then we can again show that the set of reducible polynomials in $\mathcal{P}_{d;n_1}$ has codimension at least $k$. Then arguments as in 1) yield the conclusions. For the sufficiency of $d\geq \max \{2k-3,7\}$, please see Step 1 in the proof of Theorem \ref{Theorem2}.  

\end{proof}

\begin{definition}
A generic $\mathbb{C}^k\subset \mathbb{C}^n$ is defined to be a generic complex subspace of $\mathbb{C}^n$ of dimension $k$. In other words, it is the image of a generic linear map $A :\mathbb{C}^k\rightarrow \mathbb{C}^n$. Another way to say is that it is a generic point of the Grassmanian manifold $Gr_k(\mathbb{C}^n)$.  
\end{definition}

Here is an extension for the estimate in higher dimensions.  

\begin{theorem} 
Let $d\geq 2$ (for $k\geq 3$, we require  $d\geq \max \{2k-3,7\}$). 
Then there is a constant $C_d>0$ depending only on $d$ such that the following holds. 
Let $F$ be a generic polynomial on $\mathbb{C}^{N_1}\times \ldots \times \mathbb{C}^{N_k}$ of degree $d$. 

For $j=1,\ldots ,k$, let $P_j\subset \mathbb{C}^{N_j}$ be a codimension $2$ subvariety. 
Let $W$ be the union of components of $P_1\times \ldots \times P_k$ which belong to $\{F=0\}$. 
Then 
$$\deg (W)\leq C_d \deg (P_1)^{2/3}\deg (P_2)^{2/3} \deg (P_3)\ldots \deg (P_k).$$

\label{Theorem2}
\end{theorem}

\begin{proof}

The proof proceeds in several steps. 
We begin by explaining the overall strategy in more detail. 
We will first restrict $F$ to a generic $\mathbb{C}^2\times \ldots \times \mathbb{C}^2$ 
inside $\mathbb{C}^{N_1}\times \ldots \times \mathbb{C}^{N_k}$, and then further restrict 
to the first $2$ factors $\mathbb{C}^2\times \mathbb{C}^2$. 
The goal is to reduce the problem to a setting where known results on Cartesian polynomials apply. 
We will show that the restriction is non-Cartesian, and hence can apply the results from the Cartesian polynomial paper. 

For the purpose of clearly describing the genericity condition in the theorem, 
we will in fact proceed in the reverse order, first analyzing the two-variable situation, 
and then lifting the result back to the higher-dimensional setting. 

\medskip

Let $\mathcal{C}_d\subset \mathcal{P}_{2,2;d}$ be the closure of the set of Cartesian polynomials. 
By Proposition \ref{Proposition1}, $\mathcal{C}_d$ is of codimension at least $1$. 
However, for the purpose of this proof, we will need the following stronger statement, 
which ensures that the bad set is sufficiently small.

\medskip

\noindent
\textbf{Claim 1:} For $d\geq 7$, the codimension of $\mathcal{C}_d$ is at least $d-1$. 

\medskip

\noindent
\emph{Proof of Claim 1:} 
To prove this claim, it suffices to check that the set of polynomials in 
$\mathcal{P}_{2,2;d}$ which do not satisfy either Property $1$, $2$, or $3$ 
in Theorem \ref{TheoremSufficientConditions} has codimension at least $d-1$. 
We consider each property in turn.

\medskip

\noindent
\emph{Property 1:} 
The set of polynomials which belong to $\mathbb{C}[x,y,z,t]$ has dimension $\binom{4+d}{4}$, 
while the set of polynomials which belong to $\mathbb{C}[z,t]$ has dimension $\binom{2+d}{2}$. 
Hence the codimension is clearly at least $d-1$. 
The claim concerning irreducibility in this case follows directly from the analysis of Property 2.

\medskip

\noindent
\emph{Property 2:} 
By the proof of part 2 of Theorem \ref{TheoremGenericPolynomials}, 
the set of polynomials which do not satisfy Property 2 
(i.e. those $F(x,y,z,t)$ for which there are infinitely many $(z_0,t_0)$ 
so that the function $(x,y)\mapsto F(x,y,z_0,t_0)$ is reducible)  
has codimension at least 
\[
\min_{d_1,d_2\geq 2,\, d_1+d_2=d} d_1d_2\geq d-1.
\]

\medskip

\noindent
\emph{Property 3:} 
Here we need to show that the set of polynomials $F(x,y,z,t)$ 
for which there exist $z_1,t_1,z_2,t_2,\lambda \in \mathbb{C}$ 
such that $F(x,y,z_1,t_1)-\lambda F(x,y,z_2,t_2)$ is identically zero 
(as a function of $x,y$) has codimension at least $d-1$. 
To this end, we consider two cases.

\medskip

\noindent
\emph{Case 1: $\lambda \neq 1$.} 
This implies that the homogeneous part of degree $d$ in $x,y$ 
of a bad polynomial $F$ must be identically $0$. 
The set of such bad polynomials therefore satisfies at least $d+1$ independent linear conditions, 
and hence has codimension at least $d+1$.

\medskip

\noindent
\emph{Case 2: $\lambda =1$.} 
In this case, if we write 
$$F(x,y,z,t)=\sum_{i+j\leq d}F_{i,j}(z,t)x^iy^j,$$
where $F_{i,j}(z,t)$ is a polynomial of degree $\leq d-(i+j)$, 
then we have 
$$F_{i,j}(z_1,t_1)=F_{i,j}(z_2,t_2)$$
for all $(i,j)$. 

We note that there are $d(d-1)/2$ such pairs $(i,j)$ with $i+j\leq d-1$, 
for which a generic $F_{i,j}(z,t)$ has degree at least $1$. 
Choose $m$ to be the largest integer such that $m\leq d(d-1)/6$, 
and partition the set of indices $(i,j)$ into at least $m$ disjoint triples 
$(i_1,j_1),(i_2,j_2),(i_3,j_3)$. 

Using that the set of triples $(F_1,F_2,F_3)\in \mathcal{P}_{2;d_1}\times \mathcal{P}_{2;d_2}\times \mathcal{P}_{2;d_3}$ 
(with $d_1,d_2,d_3\geq 1$) having a common pair  $(z_1,t_1)$ and $(z_2,t_2)$ for which $F_j(z_1,t_1)=F_j(z_2,t_2)$ for all $j=1,2,3$ has codimension at least $1$, which can be seen by assuming first $(z_1,t_1)=(0,0)$ and reducing to that  $\widehat{F_j}(z,t)=F_j(z,t)-F_j(0,0)$ having a common root $(z_2,t_2)\not= (0,0)$, 
we conclude that the set of bad polynomials $F$ in Case 2 has codimension at least $m$, 
which is at least $d-1$ under our assumption.

\medskip

(Q.E.D.)

\medskip

We note that if $k=2$ then Step 1 below is not needed. 

\medskip

\noindent
\textbf{Step 1:} 
Let $k\geq 3$ and $d\geq \max \{2k-3,7\}$. 
Let $G$ be a generic polynomial on $(\mathbb{C}^2)^k$. 
Then there is a finite set $A\subset (\mathbb{C}^2)^{k-2}$ 
(the last $k-2$ factors), where $A$ depends on $G$, 
such that if $(z_3^{(0)},\ldots , z_k^{(0)})\notin A$ 
then the function $(z_1,z_2)\mapsto F(z_1,z_2,z_3^{(0)},\ldots , z_k^{(0)})$ 
is non-Cartesian. 

\medskip

\noindent
\emph{Proof of Step 1:} 
We consider the evaluation map 
\[
\psi : \mathcal{P}_{2,\ldots ,2;d}\times (\mathbb{C}^2)^{k-2}\rightarrow \mathcal{P}_{2,2;d},
\]
given by $(F,z_3,\ldots ,z_k)\mapsto$ the function 
$(z_1,z_2)\mapsto F(z_1,z_2,z_3,\ldots ,z_k)$. 

As in the proof of part 2 of Theorem \ref{TheoremGenericPolynomials}, 
this map $\psi$ is equi-dimensional, 
meaning that the preimage of any point in $\mathcal{P}_{2,2;d}$ is non-empty 
and has constant dimension. 

By Claim 1, the bad set $\mathcal{C}_d\subset \mathcal{P}_{2,2;d}$ 
has codimension at least $d-1\geq 2(k-2)$. 
Hence its preimage 
\[
\mathcal{C}_{d;k}=\psi^{-1}(\mathcal{C}_d)
\]
also has codimension at least $d-1\geq 2(k-2)$. 
Note that $2(k-2)$ is precisely the dimension of $(\mathbb{C}^2)^{k-2}$. 

Therefore, for a generic $F\in \mathcal{P}_{2,\ldots ,2;d}$, 
the set 
\[
(\{F\}\times (\mathbb{C}^2)^{k-2})\cap \mathcal{C}_{d;k}
\]
is at most finite. 

\medskip

Since this intersection is defined by algebraic conditions whose degrees are bounded in terms of $d$, 
standard arguments imply that its cardinality is bounded by a constant depending only on $d$. 

\medskip

This concludes the proof of Step 1. 

(Q.E.D.)

\medskip

\noindent
\textbf{Step 2:} 
Let $d\geq 2$ if $k=2$, and let $d\geq \max \{2k-3,7\}$ if $k\geq 3$. 
There is a constant $C_d>0$ depending only on $d$ such that the following holds. 
Let $F$ be a generic polynomial on $\mathbb{C}^{N_j=2}\times \ldots \times \mathbb{C}^{N_k=2}$ of degree $d$. 

For $j=1,\ldots ,k$, let $P_j\subset \mathbb{C}^{2}$ be a finite set. 
Let $W$ be the intersection of $P_1\times \ldots \times P_k$ and $\{F=0\}$. 

Then 
$$\sharp W\leq C_d  (\sharp P_1)^{2/3} (\sharp P_2)^{2/3}  (\sharp P_3)\ldots  (\sharp P_k).$$

\medskip

\noindent
\emph{Proof of Step 2:} 
If $k=2$, then this is simply the result from the Cartesian polynomials paper. 

Now we consider $k\geq 3$. Let $F$ be a generic polynomial which satisfies the conclusion of Step 1. 
Let $A$ be the associated exceptional finite set. 

For each $(z_3^{(0)},\ldots , z_k^{(0)})\in P_3\times \ldots \times P_k\setminus A$, 
the polynomial 
\[
(z_1,z_2)\mapsto F_{z_3^{(0)},\ldots , z_k^{(0)}}(z_1,z_2)
=F(z_1,z_2,z_3^{(0)},\ldots , z_k^{(0)})
\]
is non-Cartesian. Therefore, by the Cartesian polynomials result,  
\[
\sharp (\{F_{z_3^{(0)},\ldots , z_k^{(0)}}=0\})\cap (P_1\times P_2)
\leq C_d  (\sharp P_1)^{2/3} (\sharp P_2)^{2/3}.
\]

Taking the sum over all $(z_3^{(0)},\ldots , z_k^{(0)})$ 
we obtain the conclusion of Step 2. 

\medskip

(Q.E.D.)

\medskip

\noindent
\textbf{Step 3:} 

Let $d\geq 2$ if $k=2$, and let $d\geq \max \{2k-3,7\}$ if $k\geq 3$. 
For a generic polynomial $F$ on $\mathbb{C}^{N_1}\times \ldots \times \mathbb{C}^{N_k}$ of degree $d$, 
the restriction of $F$ to a generic $\mathbb{C}^{2}\times \ldots \times \mathbb{C}^{2}$ 
satisfies the conclusions of Step 1. 

\medskip

\noindent
\emph{Proof of Step 3:} 
As in the above arguments, the restriction of $F$ to such a subspace can be represented 
by the following algebraic map $\phi$. 

Let 
\[
\phi : \mathcal{P}_{N_1,\ldots ,N_k;d}\times 
(Gr_2(\mathbb{C}^{N_1})\times \ldots \times Gr_2(\mathbb{C}^{N_k}))
\rightarrow \mathcal{P}_{2,\ldots ,2;d}
\]
be the evaluation map.

Let $\mathcal{C}_{d;k}$ be the bad set from Step 1, 
and let $\mathcal{B}_{d;k}$ be its projection. 
Then the restriction is good provided the image avoids $\mathcal{B}_{d;k}$. 

Since $\mathcal{B}_{d;k}$ is a proper subvariety, 
it suffices to show that $\phi$ is equi-dimensional. 

We now verify this. 
For each fixed $(A_1,\ldots ,A_k)$, the induced map $\phi_{A_1,\ldots ,A_k}$ is linear and surjective. 
Hence all fibers have the same dimension $=$ $\dim ( \mathcal{P}_{N_1,\ldots ,N_k;d})-\dim \mathcal{P}_{2,\ldots ,2;d}$, which is independent of $(A_1,\ldots ,A_k)$,  and therefore $\phi$ is equi-dimensional.

(Q.E.D.)

\medskip

Now we complete the proof of Theorem \ref{Theorem2}. 

We let $F$ be a generic polynomial on $\mathbb{C}^{N_1}\times \ldots \times \mathbb{C}^{N_k}$ 
so that the conclusions of Step 3 hold. 

For $j=1,\ldots ,k$, let $P_j\subset \mathbb{C}^{N_j}$ be codimension $2$ subvarieties, 
and let $W$ be as defined above. 

For a generic $\mathbb{C}^{2}\times \ldots \times \mathbb{C}^{2}$, 
let $G$ be the restriction of $F$. 
Then $G$ satisfies the conclusions of Step 1. 

Moreover, by B\'ezout's theorem, 
\[
(P_1\times \ldots P_k)\cap (\mathbb{C}^2)^k
=\widehat{P}_1\times \ldots \times \widehat{P}_k,
\]
where each $\widehat{P}_j\subset \mathbb{C}^2$ is a finite set 
whose cardinality equals $\deg (P_j)$. 

Similarly, $W\cap (\mathbb{C}^2)^k$ is finite 
and has cardinality equal to $\deg (W)$. 
Moreover,
\[
W\cap (\mathbb{C}^2)^k
=\{G=0\}\cap (\widehat{P}_1\times \ldots \times \widehat{P}_k).
\]

Therefore, applying Step 2 to $G$ and the finite sets $\widehat{P}_j$, 
we obtain the desired bound. This completes the proof.

\end{proof}

In the next section we present  some examples which show that the above theorems are sharp. Here is a remark concerning the proof of Theorem \ref{Theorem2}. We note that Step 1 in the proof of Theorem \ref{Theorem2} does not hold if the polynomial is not generic.  We construct now a counterexample on $\mathbb{C}^2\times \mathbb{C}^2 \times \mathbb{C}^2$. 

\medskip

Let
\[
F(x_1,x_2,y_1,y_2,z_1,z_2)
=
x_1y_1 + y_2z_1 + z_2x_2 + 1.
\]
Fix $z_1 = 0$, and let $z_2 \neq 0$. Then
\[
F(x,y,z') = x_1y_1 + z_2 x_2 + 1.
\]
This can be written as
\[
F(x,y,z')
=
(z_2 x_2 + 1)\cdot 1
+
y_1 \cdot x_1,
\]
which is of the form
\[
G(x)F_1(x,y) + H(y)F_2(x,y),
\]
and hence is Cartesian in the $(x,y)$ variables.

\medskip

On the other hand, one can verify (by using Theorem \ref{TheoremSufficientConditions}) that the original polynomial $F$ is non-Cartesian on 
\[
\mathbb{C}^2 \times \mathbb{C}^2 \times \mathbb{C}^2.
\]

\section{Examples for sharpness of intersection estimates}

Here we present many examples which illustrate the sharpness of the estimates in the main theorems of this paper, as well as the subtlety when extending the known estimates from $\mathbb{C}^2\times \mathbb{C}^2$ to higher dimensions. For simplicity, we present here examples on $ \mathbb{C}^{n_1}\times \mathbb{C}^{n_2}$ and $\mathbb{C}^{n_1}\times \mathbb{C}^{n_2}\times \mathbb{C}^{n_3}$ only, but the arguments can be adapted to other cases.  




\subsection{Examples for sharpness of estimates in $\mathbb{C}^2\times \mathbb{C}^2$} In the work of Mojarrad, Pham, Valculescu and Zeeuw, some examples illustrating the sharpness of their theorem \cite{MojarradPhamValculescuZeeuw} were given. Here we present an example, arising in the proof of Proposition \ref{Proposition1}, which will be utilised later. 

We construct explicit finite sets $P\subset\mathbb C^2_{x,y}$ and 
$Q\subset\mathbb C^2_{a,b}$ with $|P|\sim|Q|\sim n$ such that
\[
|Z(f)\cap(P\times Q)|\sim n^{4/3}, 
\]

where $f(x,y,a,b)=(x-a)^2+(y-b)^2+H(a,b)$ and $H(a,b)=1-a^2-b^2$. From the proof of Proposition \ref{Proposition1}, it follows that $f$ is non-Cartesian. 

We have
\begin{align*}
f(x,y,a,b)
&=(x-a)^2+(y-b)^2+1-a^2-b^2\\
&=x^2-2ax+a^2+y^2-2by+b^2+1-a^2-b^2\\
&=x^2+y^2+1-2ax-2by.
\end{align*}

Therefore $f(x,y,a,b)=0$ is equivalent to
\begin{equation}\label{eq:main_linear}
2ax+2by = x^2+y^2+1,
\end{equation}
or equivalently,
\begin{equation}\label{eq:linear_relation}
ax+by=\frac{x^2+y^2+1}{2}.
\end{equation}

Note that for each fixed $(x,y)\in\mathbb C^2$ with $y\neq 0$.  
Solving \eqref{eq:linear_relation} for $b$, we obtain
\begin{align*}
by &= \frac{x^2+y^2+1}{2} - ax,\\
b &= -\frac{x}{y}a + \frac{x^2+y^2+1}{2y}.
\end{align*}

Thus for each $(x,y)$ with $y\neq 0$, the equation
$f(x,y,a,b)=0$ defines a line in the $(a,b)$-plane:
\begin{equation}\label{eq:line_equation}
b = -\frac{x}{y}\,a + \frac{x^2+y^2+1}{2y}.
\end{equation}

Now given an integer parameter $s\ge1$ and define
\[
A=\{1,2,\dots,s\},
\qquad
B=\{1,2,\dots,2s^2\}.
\]
Let
\[
Q:=A\times B \subset \mathbb C^2_{a,b}.
\]
Then
\[
|Q|=|A||B|=2s^3.
\]

Define a family of lines
\[
\mathcal L
=
\{\, b=ma+t:\ m\in\{1,\dots,s\},\ t\in\{1,\dots,s^2\}\,\}.
\]
Hence
\[
|\mathcal L|=s\cdot s^2 = s^3.
\]

For each line $b=ma+t$ and each $a\in A$,
\[
b = ma+t \in \{1,\dots,2s^2\},
\]
so every line of $\mathcal L$ contains exactly $s$ points of $Q$.

Therefore the total number of point--line incidences is
\begin{equation}\label{eq:incidence_count}
I(Q,\mathcal L)
=
|\mathcal L|\cdot s
=
s^3\cdot s
=
s^4.
\end{equation}

Therefore for fixed $m\in\{1,\dots,s\}$ and $t\in\{1,\dots,s^2\}$.  
We construct $(x,y)$ such that the line in \eqref{eq:line_equation} becomes
\[
b=ma+t.
\]

First, the slope condition says:
\begin{equation}\label{eq:slope_condition}
-\frac{x}{y}=m
\quad\Longleftrightarrow\quad
x=-my.
\end{equation}

Substitute $x=-my$ into the intercept term:
\begin{align*}
\frac{x^2+y^2+1}{2y}
&=
\frac{m^2y^2+y^2+1}{2y}\\
&=
\frac{(m^2+1)y^2+1}{2y}.
\end{align*}

To obtain intercept $t$, we require
\[
\frac{(m^2+1)y^2+1}{2y}=t,
\]
which is equivalent to the quadratic equation
\begin{equation}\label{eq:quadratic_y}
(m^2+1)y^2 - 2ty + 1 = 0.
\end{equation}

Since we are working over $\mathbb C$, equation \eqref{eq:quadratic_y} always has a solution $y\in\mathbb C$.

For any such solution $y$, define
\[
x:=-my.
\]

Then the corresponding line in \eqref{eq:line_equation} is precisely
\[
b=ma+t.
\]

\medskip

Define
\[
P:=\{(x_{m,t},y_{m,t}) : 1\le m\le s,\ 1\le t\le s^2\}.
\]
Then
\[
|P|=s^3,
\]
and the incidences between $P$ and $Q$ via $f=0$ coincide exactly with the point--line incidences counted in \eqref{eq:incidence_count}.

Hence
\[
|Z(f)\cap(P\times Q)| = s^4.
\]

Since $|P|\sim |Q|\sim s^3$, writing $n=s^3$ gives
\[
|Z(f)\cap(P\times Q)| \sim n^{4/3}.
\]

\begin{remark}
The special case 
\[
H(a,b) \equiv -1
\]
deserves separate mention. In this case the polynomial becomes
\[
f(x,y,a,b) = (x-a)^2 + (y-b)^2 - 1,
\]
and the equation $f(x,y,a,b)=0$ is equivalent to
\[
(x-a)^2 + (y-b)^2 = 1,
\]
that is, $(x,y)$ lies on the unit circle centered at $(a,b)$.

Thus, for finite sets $P,Q \subset \mathbb C^2$ (or $\mathbb R^2$), the quantity
\[
|Z(f) \cap (P \times Q)|
\]
counts the number of pairs $(p,q) \in P \times Q$ such that the Euclidean distance between $p$ and $q$ equals $1$. In particular, when $P=Q \subset \mathbb R^2$, this reduces to the classical Erd\H{o}s unit distance problem, which asks for the maximum number of unit distances determined by $n$ points in the plane which is still an open problem.

Therefore, the constant case $H \equiv -1$ lies at the heart of a central and notoriously difficult problem in combinatorial geometry. Our analysis of the general family
\[
f(x,y,a,b) = (x-a)^2 + (y-b)^2 + H(a,b)
\]
may be viewed as a natural deformation of the unit distance setting.
\end{remark}

\subsection{Failure of naive intersection estimates  for non-Cartesian polynomials in higher dimensions}

The polynomial $F(x,y,z,a,b,c)=z-ax-by-c$  satisfies Properties (1), (2), and (3) in Theorem \ref{TheoremSufficientConditions}, and hence is non-Cartesian by that theorem. 

However in six variables, being non-Cartesian is not sufficient to get a nontrivial upper bound for the naive intersecton estimates. Consider the above polynomial as defined on $\mathbb{C}^3\times \mathbb{C}^3$. We show that there exist finite sets $P,Q\subset \mathbb C^3$ with
$|P|=|Q|=n$ such that
\[
|Z(F)\cap (P\times Q)| \sim n^2.
\]

\medskip
\noindent

First define
\[
P := \{\, p_i=(i,0,0) : i=1,2,\dots,n \,\}
\subset \mathbb C^3_{x,y,z},
\]
and
\[
Q := \{\, q_j=(0,j,0) : j=1,2,\dots,n \,\}
\subset \mathbb C^3_{a,b,c}.
\]

Observe that all points in $P$ are distinct, since their first
coordinates $i$ are distinct. Similarly, all points in $Q$
are distinct, since their second coordinates $j$ are distinct.
Therefore
\[
|P|=n,
\qquad
|Q|=n.
\]

\medskip
\noindent

Let $p_i=(i,0,0)\in P$ and $q_j=(0,j,0)\in Q$.
Substituting these coordinates into $F$, we obtain
\begin{align*}
F(p_i,q_j)
&= F(i,0,0,\,0,j,0) \\
&= z - ax - by - c \\
&= 0 - (0)\cdot i - j\cdot 0 - 0 \\
&= 0.
\end{align*}

Thus for every pair $(i,j)$ with
$1\le i\le n$ and $1\le j\le n$,
we have
\[
F(p_i,q_j)=0.
\]
In other words,
\[
P\times Q \subset Z(F).
\]

\medskip
\noindent

Since $Z(F)\cap (P\times Q)$ is a subset of $P\times Q$
and we have already shown that $P\times Q\subset Z(F)$,
it follows that
\[
Z(F)\cap (P\times Q) = P\times Q.
\]

Therefore,
\[
|Z(F)\cap (P\times Q)|
=
|P\times Q|
=
|P|\,|Q|
=
n\cdot n
=
n^2.
\]

\begin{remark}[Higher-dimensional phenomena]
The intersection behavior in higher dimensions is substantially more subtle 
than in the $\mathbb C^2 \times \mathbb C^2$ setting.

In the $\mathbb C^2 \times \mathbb C^2$ case, once $f$ is non-Cartesian, the zero set 
$Z(f)$ cannot contain a large product subvariety of the form 
$V \times W$ with $\dim V = \dim W = 1$. 
As a consequence, the problem can be reduced to a point--curve 
(or point--line) incidence configuration, and one obtains the sharp 
$n^{4/3}$ upper bound.

In contrast, in the $\mathbb C^3 \times \mathbb C^3$ setting, even if $f$ is 
non-Cartesian, the hypersurface $Z(f)$ may still contain 
large positive-dimensional subvarieties which are not of pure product type. 
These subvarieties can support dense incidences and may lead to 
quadratic behavior for suitably chosen sets $P$ and $Q$, as illustrated 
by the preceding examples.

Therefore, in higher dimensions, the non-Cartesian condition alone is 
insufficient to guarantee a nontrivial upper bound such as $O(n^{3/2})$. 
To obtain meaningful incidence estimates in $\mathbb C^3 \times \mathbb C^3$, 
one must impose additional structural assumptions on $f$ 
that rule out the presence of large subvarieties inside $Z(f)$.
\end{remark}













\subsection{Examples for sharpness of estimates in higher dimensions}

\subsubsection{Sharpness estimates via lifting}
We start with a simple but useful observation showing that
non-Cartesianity is stable under a trivial extension of variables.

\begin{lemma}[Lifting preserves non-Cartesianity]
\label{lem:lifting_non_cartesian}
Let
\[
F(x,y)\in \mathbb C[x_1,x_2,y_1,y_2]
\]
be a polynomial that is non-Cartesian in the two-block sense. Define
\[
f(x,y,z):=F(x,y),
\qquad
(x,y,z)\in (\mathbb C^2)^3.
\]
Then \(f\) is non-Cartesian in the three-block sense.
\end{lemma}

\begin{proof}
Assume, for contradiction, that \(f\) is Cartesian in the three-block sense.
Then there exist non-constant polynomials
\[
G_1(x)\in\mathbb C[x_1,x_2],\quad
G_2(y)\in\mathbb C[y_1,y_2],\quad
G_3(z)\in\mathbb C[z_1,z_2],
\]
and polynomials \(H_1,H_2,H_3\in\mathbb C[x,y,z]\) such that
\begin{equation}\label{eq:3block_cart}
f(x,y,z)
=
G_1(x)H_1(x,y,z)
+
G_2(y)H_2(x,y,z)
+
G_3(z)H_3(x,y,z).
\end{equation}
Since \(f(x,y,z)=F(x,y)\) is independent of \(z\), \eqref{eq:3block_cart}
implies
\begin{equation}\label{eq:F_in_ideal}
F(x,y)
\in
\bigl(G_1(x),\,G_2(y),\,G_3(z)\bigr)
\subset \mathbb C[x,y,z].
\end{equation}

Because \(G_3\) is non-constant, there exists \(z_0\in\mathbb C^2\) such that
\(G_3(z_0)=0\). Evaluating \eqref{eq:F_in_ideal} at \(z=z_0\), we obtain
\[
F(x,y)
=
G_1(x)H_1(x,y,z_0)
+
G_2(y)H_2(x,y,z_0),
\]
which shows that
\[
F(x,y)\in\bigl(G_1(x),\,G_2(y)\bigr)\subset\mathbb C[x,y].
\]
Thus \(F\) is Cartesian in the two-block sense, contradicting the assumption.
\end{proof}

\medskip

We now apply this lifting principle to transfer sharp incidence lower bounds
from the two-block setting to the three-block setting.

Let
\[
f:\ (\mathbb C^2)^3 \to \mathbb C
\]
be a polynomial, and denote its zero set by
\[
Z(f):=\{(x,y,z)\in(\mathbb C^2)^3:\ f(x,y,z)=0\}.
\]
For finite sets
\[
P,Q,R\subset\mathbb C^2,
\]
we consider the quantity
\[
|Z(f)\cap(P\times Q\times R)|.
\]

\begin{proposition}[Sharp lower bound via lifting]
\label{prop:3block_lifting_lower}
Let \(F(x,y)\in\mathbb C[x_1,x_2,y_1,y_2]\) be a non-Cartesian polynomial.
Assume that there exist finite sets \(P,Q\subset\mathbb C^2\) with
\(|P|=|Q|=n\) such that
\[
|Z(F)\cap(P\times Q)|\sim n^{4/3}.
\]
Define
\[
f(x,y,z):=F(x,y),
\qquad (x,y,z)\in(\mathbb C^2)^3.
\]
Then \(f\) is non-Cartesian in the three-block sense, and for any finite set
\(R\subset\mathbb C^2\) with \(|R|=n\),
\[
|Z(f)\cap(P\times Q\times R)|
=
|Z(F)\cap(P\times Q)|\cdot |R|
\sim n^{7/3}.
\]

Moreover, for every finite sets $P_1,Q_1,R_1\subset \mathbb{C}^2$ with $|P_1|=|Q_1|=|R_1|=n$ we have
\[
|Z(f)\cap(P_1\times Q_1\times R_1)|
=
|Z(F)\cap(P_1\times Q_1)|\cdot |R_1|
\lesssim n^{7/3}.
\]
\end{proposition}

\begin{proof}
Since \(f(x,y,z)=F(x,y)\) is independent of \(z\), we have
\[
Z(f)=Z(F)\times \mathbb C^2.
\]
Therefore,
\[
Z(f)\cap(P\times Q\times R)
=
\bigl(Z(F)\cap(P\times Q)\bigr)\times R,
\]
and hence
\[
|Z(f)\cap(P\times Q\times R)|
=
|Z(F)\cap(P\times Q)|\cdot |R|.
\]
Substituting \(|Z(F)\cap(P\times Q)|\sim n^{4/3}\) and \(|R|=n\) yields
the desired bound. The claim for $P_1,Q_1,R_1$ can be proven similarly. The non-Cartesianity of \(f\) follows from
Lemma~\ref{lem:lifting_non_cartesian}.
\end{proof}

\subsubsection{A truly higher dimensional example on $\mathbb{C}^2\times \mathbb{C}^2\times \mathbb{C}^2$} The examples provided by Proposition \ref{prop:3block_lifting_lower} are of lower dimensional nature. In this section we present a truly higher dimensional example, also in arbitrarily large degrees. 

\begin{proposition}[A high-degree example with all fibres non-Cartesian]
Let
\[
F\bigl((x_1,x_2),(y_1,y_2),(z_1,z_2)\bigr)
=
x_2^{\,d}+x_1^2+y_1x_1+y_2+z_1,
\]
where $d\ge 3$ is an odd integer. Then for every fixed
\[
z=(z_1,z_2)\in \mathbb{C}^2,
\]
the polynomial
\[
F_z(x,y):=F(x,y,z)
=
x_2^{\,d}+x_1^2+y_1x_1+y_2+z_1
\]
is non-Cartesian as a polynomial on
\[
\mathbb{C}^2_x\times \mathbb{C}^2_y.
\]
In particular, the set of exceptional parameters $z\in \mathbb{C}^2$ for which
$F_z$ is Cartesian is empty. 

As a consequence:  

i) For every $P_1,P_2,P_3\subset \mathbb{C}^2$ finite sets and  $|P_1|=|P_2|=|P_3|$, let $W_1=\{F=0\}\cap (P_1\times P_2\times P_3)$ then 
$$|W_1|\lesssim n^{7/3}.$$

ii) There are $P_1,P_2,P_3\subset \mathbb{C}^2$ finite sets with  $|P_1|=|P_2|=|P_3|$, so that if $W_1=\{F=0\}\cap (P_1\times P_2\times P_3)$, then  
$$ |W_1|\sim n^{7/3}.$$

\end{proposition}

\begin{proof}
Fix
\[
z=(z_1,z_2)\in \mathbb{C}^2,
\]
and write
\[
c:=z_1\in \mathbb{C}.
\]
Set
\[
H_c(x,y):=x_2^{\,d}+x_1^2+y_1x_1+y_2+c.
\]
Thus
\[
F_z(x,y)=H_c(x,y).
\]
We will show that $H_c$ is non-Cartesian for every $c\in \mathbb{C}$ by verifying
the hypotheses of Theorem~4.1.

Recall that in the notation of Theorem~4.1, the first block is
\[
x=(x_1,x_2)\in \mathbb{C}^2,
\]
and the second block is
\[
y=(y_1,y_2)\in \mathbb{C}^2.
\]

\medskip

\noindent
\textbf{Step 1. $H_c$ is irreducible and $H_c\notin \mathbb{C}[y_1,y_2]$.}

The second assertion is immediate: the polynomial $H_c$ depends nontrivially on
$x_1$ and $x_2$, and therefore does not belong to $\mathbb{C}[y_1,y_2]$.

It remains to prove irreducibility of $H_c$ in
\[
\mathbb{C}[x_1,x_2,y_1,y_2].
\]

We regard $H_c$ as a polynomial in the single variable $x_2$ over the field
\[
K:=\mathbb{C}(x_1,y_1,y_2).
\]
Then
\[
H_c=x_2^{\,d}+a,
\qquad
a:=x_1^2+y_1x_1+y_2+c\in K.
\]
Thus $H_c$ is a polynomial of the form
\[
T^d+a\in K[T].
\]

We claim that $T^d+a$ is irreducible over $K$. Since $d$ is odd, every prime divisor
$p$ of $d$ is odd. By the standard irreducibility criterion for binomials over a field
of characteristic $0$, it suffices to check that $a$ is not a $p$-th power in $K$
for any prime $p\mid d$.

Assume for contradiction that for some prime $p\mid d$ there exists
\[
r(x_1,y_1,y_2)\in K
\]
such that
\[
a=r^p.
\]
Write
\[
r=\frac{u}{v},
\]
where $u,v\in \mathbb{C}[x_1,y_1,y_2]$ are coprime and $v\neq 0$. Then
\[
a=\frac{u^p}{v^p},
\]
so
\[
a\,v^p=u^p.
\]
Since $a=x_1^2+y_1x_1+y_2+c$ is a polynomial and $\gcd(u,v)=1$, it follows that
$v^p$ divides $u^p$ in the UFD $\mathbb{C}[x_1,y_1,y_2]$, hence $v$ divides $u$.
By coprimality, $v$ must be a nonzero constant. Therefore $r\in \mathbb{C}[x_1,y_1,y_2]$,
and so $a$ is a $p$-th power in the polynomial ring.

Now view both sides as polynomials in the variable $x_1$ with coefficients in
\[
\mathbb{C}[y_1,y_2].
\]
The polynomial
\[
a=x_1^2+y_1x_1+y_2+c
\]
has degree $2$ in $x_1$. If
\[
a=r^p,
\]
then the degree in $x_1$ of the right-hand side is divisible by $p$.
Since $p$ is odd and $2$ is not divisible by $p$, this is impossible.

Therefore $a$ is not a $p$-th power in $K$ for any prime divisor $p$ of $d$.
Hence $T^d+a$ is irreducible over $K$, and so $H_c$ is irreducible in
\[
\mathbb{C}[x_1,x_2,y_1,y_2].
\]

This proves condition~(1) in Theorem~4.1.

\medskip

\noindent
\textbf{Step 2. For every $(y_1,y_2)\in \mathbb{C}^2$, the specialized polynomial in
$x$ is irreducible.}

Fix arbitrary
\[
(\alpha,\beta)\in \mathbb{C}^2.
\]
Then the specialization of $H_c$ at $(y_1,y_2)=(\alpha,\beta)$ is
\[
H_c(x,\alpha,\beta)
=
x_2^{\,d}+x_1^2+\alpha x_1+\beta+c.
\]
We must show that this polynomial is irreducible in
\[
\mathbb{C}[x_1,x_2].
\]

Again, regard it as a polynomial in the single variable $x_2$ over the field
\[
L:=\mathbb{C}(x_1).
\]
Then
\[
H_c(x,\alpha,\beta)=x_2^{\,d}+b,
\qquad
b:=x_1^2+\alpha x_1+\beta+c\in L.
\]

As in Step~1, it suffices to show that $b$ is not a $p$-th power in $L$ for any prime
$p\mid d$. Suppose
\[
b=s^p
\]
for some $s\in L$. Writing $s=u/v$ with coprime polynomials
\[
u,v\in \mathbb{C}[x_1],
\]
the same coprimality argument as before shows that in fact $v$ must be constant,
hence $s\in \mathbb{C}[x_1]$ and $b$ is a $p$-th power in $\mathbb{C}[x_1]$.

But
\[
b=x_1^2+\alpha x_1+\beta+c
\]
has degree $2$ in $x_1$, while a $p$-th power in $\mathbb{C}[x_1]$ has degree divisible
by $p$. Since $p$ is odd, this is impossible.

Therefore $b$ is not a $p$-th power in $L$ for any prime $p\mid d$, and hence
\[
x_2^{\,d}+b
\]
is irreducible over $L$, and therefore irreducible in $\mathbb{C}[x_1,x_2]$.

We conclude that for every $(\alpha,\beta)\in \mathbb{C}^2$, the specialized polynomial
\[
x\longmapsto H_c(x,\alpha,\beta)
\]
is irreducible in $\mathbb{C}[x_1,x_2]$.

Hence the set
\[
\left\{(\alpha,\beta)\in \mathbb{C}^2:
H_c(x,\alpha,\beta)\text{ is reducible in }\mathbb{C}[x_1,x_2]\right\}
\]
is empty, and in particular finite.

Thus condition~(2) of Theorem~4.1 holds.

\medskip

\noindent
\textbf{Step 3. There are no nontrivial scalar coincidences among the fibres in the
$y$-variables.}

Fix
\[
(\alpha,\beta)\in \mathbb{C}^2.
\]
Suppose there exist
\[
(\alpha',\beta')\in \mathbb{C}^2
\qquad\text{and}\qquad
\lambda\in \mathbb{C}
\]
such that
\[
H_c(x,\alpha,\beta)=\lambda\, H_c(x,\alpha',\beta')
\]
as polynomials in $(x_1,x_2)$.

Expanding both sides, this means
\[
x_2^{\,d}+x_1^2+\alpha x_1+\beta+c
=
\lambda\bigl(x_2^{\,d}+x_1^2+\alpha' x_1+\beta'+c\bigr).
\]
Now compare coefficients.

From the coefficient of $x_2^{\,d}$ we obtain
\[
1=\lambda.
\]
From the coefficient of $x_1^2$ we again obtain
\[
1=\lambda,
\]
which is consistent with the previous identity. Since $\lambda=1$, comparing the
coefficient of $x_1$ yields
\[
\alpha=\alpha',
\]
and comparing the constant term yields
\[
\beta+c=\beta'+c,
\]
hence
\[
\beta=\beta'.
\]

Therefore the only pair $(\alpha',\beta')$ for which
\[
H_c(x,\alpha,\beta)=\lambda\,H_c(x,\alpha',\beta')
\]
can hold is the trivial one
\[
(\alpha',\beta')=(\alpha,\beta),
\qquad \lambda=1.
\]

It follows that for every $(\alpha,\beta)\in \mathbb{C}^2$, the set
\[
\left\{
(\alpha',\beta')\in \mathbb{C}^2:\ \exists \lambda\in \mathbb{C}\text{ such that }
H_c(x,\alpha,\beta)=\lambda H_c(x,\alpha',\beta')
\right\}
\]
consists of exactly one point, and in particular is finite.

Thus condition~(3) of Theorem~4.1 holds.

\medskip

We have verified all three hypotheses of Theorem~4.1 for the polynomial
\[
H_c(x,y)=x_2^{\,d}+x_1^2+y_1x_1+y_2+c.
\]
Therefore Theorem~4.1 implies that $H_c$ is non-Cartesian on
\[
\mathbb{C}^2_x\times \mathbb{C}^2_y,
\]
for arbitrary $c$. 

\noindent
\textbf{Step 4. Proofs of consequences.}  

i) Let $P_1,P_2,P_3\subset \mathbb{C}^2$ be finite sets with $|P_1|=|P_2|=|P_3|=n$. For each $z\in P_3$, by the above  the polynomial $(x,y)\mapsto H(x,y,z)$ is non-Cartesian and has degree $d$. Hence, by the intersection estimates on $\mathbb{C}^2\times \mathbb{C}^2$, we have
$$|\{(x,y): H(x,y,z)=0\}\cap (P_1\times P_2)|\lesssim n^{4/3},$$
where the constant is independent of $z$. Summing up over $P_3$, we obtain
$$|\{(x,y,z): H(x,y,z)=0\}\cap (P_1\times P_2\times P_3)|\lesssim n^{7/3},$$
as wanted. 

ii) Let \(s\geq 1\) be a large integer, and set
\[
n:=s^3.
\]

We first define the set \(P_3\). Fix an arbitrary constant \(c_0\in\mathbb C\), and let
\[
P_3
:=
\{(c_0,r):1\leq r\leq s^3\}
\subset \mathbb C^2.
\]
Thus
\[
|P_3|=s^3=n.
\]

Next, define
\[
A:=\{1,2,\dots,s\},
\qquad
B:=\{1,2,\dots,2s^2\},
\]
and set
\[
P_2:=A\times B
\subset \mathbb C^2.
\]
Hence
\[
|P_2|
=
|A||B|
=
s\cdot 2s^2
=
2s^3
\sim n.
\]

We now construct the set \(P_1\).
For every pair
\[
1\leq m\leq s,
\qquad
1\leq b\leq s^2,
\]
choose a complex number \(u_{m,b}\in\mathbb C\) satisfying
\[
u_{m,b}^d
=
-m^2-b-c_0.
\]

Define
\[
p_{m,b}
:=
(-m,u_{m,b})
\in \mathbb C^2,
\]
and let
\[
P_1
:=
\{p_{m,b}:1\leq m\leq s,\ 1\leq b\leq s^2\}.
\]

We claim that all points \(p_{m,b}\) are distinct.
Indeed, if
\[
(-m,u_{m,b})
=
(-m',u_{m',b'}),
\]
then necessarily \(m=m'\). Consequently,
\[
u_{m,b}=u_{m,b'}.
\]
Taking \(d\)-th powers yields
\[
-m^2-b-c_0
=
-m^2-b'-c_0,
\]
which implies \(b=b'\).
Therefore all points are distinct, and hence
\[
|P_1|
=
s\cdot s^2
=
s^3
=
n.
\]

We now estimate the number of incidences.

Fix
\[
p_{m,b}=(-m,u_{m,b})\in P_1,
\qquad
z=(c_0,r)\in P_3.
\]
For every \(a\in A\), define
\[
y(a):=(a,ma+b)\in\mathbb C^2.
\]

We first verify that \(y(a)\in P_2\).
Since
\[
1\leq a\leq s,
\qquad
1\leq m\leq s,
\qquad
1\leq b\leq s^2,
\]
we have
\[
1
\leq
ma+b
\leq
s^2+s^2
=
2s^2.
\]
Thus
\[
y(a)\in A\times B=P_2.
\]

Next, we compute
\[
F\bigl(p_{m,b},y(a),z\bigr).
\]
Substituting
\[
x_1=-m,
\qquad
x_2=u_{m,b},
\qquad
y_1=a,
\qquad
y_2=ma+b,
\qquad
z_1=c_0,
\]
we obtain
\begin{align*}
F\bigl(p_{m,b},y(a),z\bigr)
&=
u_{m,b}^d
+
(-m)^2
+
a(-m)
+
(ma+b)
+
c_0 \\
&=
u_{m,b}^d
+
m^2
-am
+ma
+b
+
c_0 \\
&=
u_{m,b}^d+m^2+b+c_0.
\end{align*}
By the defining relation
\[
u_{m,b}^d=-m^2-b-c_0,
\]
we conclude that
\[
F\bigl(p_{m,b},y(a),z\bigr)=0.
\]

Therefore, for every choice of
\[
(m,b)\in\{1,\dots,s\}\times\{1,\dots,s^2\},
\]
and every \(a\in A\), we obtain a point
\[
\bigl(p_{m,b},y(a),z\bigr)
\in
Z(F)\cap(P_1\times P_2\times P_3).
\]

For each fixed \(z\in P_3\), there are
\[
s\cdot s^2=s^3
\]
choices of \((m,b)\), and for each such pair there are \(s\) choices of \(a\).
Hence
\[
|Z(F_z)\cap(P_1\times P_2)|
\geq
s^3\cdot s
=
s^4.
\]

Since
\[
|P_3|=s^3,
\]
we obtain
\[
|Z(F)\cap(P_1\times P_2\times P_3)|
\geq
s^4\cdot s^3
=
s^7.
\]

Finally, since \(n=s^3\), we have
\[
s^7
=
(s^3)^{7/3}
=
n^{7/3}.
\]
Therefore,
\[
|Z(F)\cap(P_1\times P_2\times P_3)|
\gtrsim
n^{7/3},
\]
as claimed.

\end{proof}

\begin{remark}
The above proposition is stronger than the fibre property used in Step~1 of the proof
of Theorem \ref{Theorem2}. There one only needs that, outside a finite exceptional set in the later
variables, the restriction to the first two blocks is non-Cartesian. In the present
example, there is in fact no exceptional set at all.
\end{remark}

\begin{remark}
The same proof works more generally for
\[
F\bigl((x_1,x_2),(y_1,y_2),(z_1,z_2)\bigr)
=
x_2^{\,d}+x_1^2+y_1x_1+y_2+\Phi(z_1,z_2),
\]
where $d\ge 3$ is odd and $\Phi\in \mathbb{C}[z_1,z_2]$ is arbitrary. For a fixed
$z\in \mathbb{C}^2$, the fiber is then
\[
x_2^{\,d}+x_1^2+y_1x_1+y_2+\Phi(z),
\]
which has exactly the same form as above with the constant $c$ replaced by $\Phi(z)$.
Hence every fiber is again non-Cartesian.
\end{remark}

\begin{remark}
The oddness of $d$ is used in the irreducibility argument. More precisely, if $p$ is a
prime divisor of $d$, then $p$ is odd, and the polynomial
\[
x_1^2+y_1x_1+y_2+c
\]
cannot be a $p$-th power in the relevant rational function field, since its degree in
$x_1$ is equal to $2$. This is the key reason why the binomial irreducibility criterion
applies so cleanly in this setting.
\end{remark}

\subsubsection{A sharpness example for Theorem~5.3 on $\mathbb{C}^{n_1}\times \mathbb{C}^{n_2}$} As usual, we use $Gr_2(\mathbb{C}^N)$ to denote the Grasmannian manifold of complex vector subspaces of $\mathbb{C}^N$ of dimension $2$.

\begin{proposition}
Let $n_1,n_2\ge 2$, and assume for simplicity that $n_1\le n_2$. Write
\[
x=(x_1,\dots,x_{n_1})\in \mathbb{C}^{n_1},\qquad
y=(y_1,\dots,y_{n_2})\in \mathbb{C}^{n_2}.
\]
Define
\[
F(x,y):=x_1^2+\cdots+x_{n_1}^2+1-2(x_1y_1+\cdots+x_{n_1}y_{n_1}).
\]
Then there exists a nonempty Zariski open subset
\[
\Omega \subset
Gr_2(\mathbb C^{n_1})
\times
Gr_2(\mathbb C^{n_2})
\]
such that for every $(U,V)\in \Omega$, the restriction
\[
F|_{U\times V}
\]
is linearly equivalent to the polynomial
\[
f(x,y,a,b)=x^2+y^2+1-2ax-2by.
\]
In particular, for every $(U,V)\in \Omega$, the sharp incidence exponent $4/3$ holds for
$F|_{U\times V}$ in the sense that there exist finite sets
\[
P\subset U,\qquad Q\subset V,\qquad |P|\sim |Q|\sim n,
\]
such that
\[
\bigl| Z(F|_{U\times V})\cap (P\times Q)\bigr| \sim n^{4/3}.
\]

As a consequence:  

i) For every $P_1\subset \mathbb{C}^{n_1}$ and $Q_1\subset \mathbb{C}^{n_2}$ effective algebraic cycles of pure codimension $2$ and $\deg (P_1)=\deg (Q_1)=n$, let $W_1$ be the union of components of $P_1\times Q_1$ which belong to the set $\{F=0\}$, then 
$$\deg (W_1)\lesssim n^{4/3}.$$

ii) There are $P_1\subset \mathbb{C}^{n_1}$ and $Q_1\subset \mathbb{C}^{n_2}$ effective algebraic cycles of pure codimension $2$ and $\deg (P_1)=\deg (Q_1)=n$, such that if $W_1$ is the union of components of $P_1\times Q_1$ which belong to the set $\{F=0\}$, then 
$$\deg (W_1)\sim n^{4/3}.$$

\label{PropositionHigherDimensionalExampleGenericSliceGood}\end{proposition}

\begin{proof}
We divide the proof into several steps. 

\medskip

\noindent
\textbf{Step 1. Writing $F$ in the form $Q(x)+1-2B(x,y)$.}

Define
\[
Q(x):=x_1^2+\cdots+x_{n_1}^2,
\]
and
\[
B(x,y):=x_1y_1+\cdots+x_{n_1}y_{n_1}.
\]
Then
\[
F(x,y)=Q(x)+1-2B(x,y).
\]
The quadratic form $Q$ is nondegenerate on $\mathbb{C}^{n_1}$, and the bilinear form
$B$ has rank $n_1$.

Let
\[
U\subset \mathbb{C}^{n_1},\qquad V\subset \mathbb{C}^{n_2}
\]
be $2$-dimensional linear subspaces. Then the restriction of $F$ to $U\times V$ is
\[
F|_{U\times V}(u,v)=Q|_U(u)+1-2\,B|_{U\times V}(u,v).
\]
Thus the restriction is always a polynomial consisting of a quadratic form in the
$U$-variables, plus a constant term, plus a bilinear coupling term between the
$U$-variables and the $V$-variables.

\medskip

\noindent
\textbf{Step 2. A generic $2$-plane $U$ sees a nondegenerate quadratic form.}

We claim that there exists a nonempty Zariski open subset
\[
\Omega_Q\subset Gr_2(\mathbb{C}^{n_1})
\]
such that for every $U\in \Omega_Q$, the restriction $Q|_U$ is nondegenerate.

Indeed, fix $U\in Gr_2(\mathbb{C}^{n_1})$, and choose a basis $u_1,u_2$ of $U$.
The restriction $Q|_U$ is represented by the symmetric Gram matrix
\[
G_Q(U)=
\begin{pmatrix}
Q(u_1,u_1) & Q(u_1,u_2)\\
Q(u_2,u_1) & Q(u_2,u_2)
\end{pmatrix},
\]
where $Q(\cdot,\cdot)$ denotes the symmetric bilinear form associated to $Q$.
The condition that $Q|_U$ be degenerate is exactly
\[
\det G_Q(U)=0.
\]
This is an algebraic condition on $U$; equivalently, it defines a Zariski closed subset
of the Grassmannian.

It remains to show that this closed subset is proper. But this is immediate, since
for example on the coordinate plane
\[
U_0:=\mbox{ The vector subspace spanned by }(e_1,e_2)\subset \mathbb{C}^{n_1},
\]
one has
\[
Q|_{U_0}=u_1^2+u_2^2,
\]
which is nondegenerate. Hence the degeneracy locus is a proper Zariski closed subset,
and therefore its complement $\Omega_Q$ is a nonempty Zariski open subset.

\medskip

\noindent
\textbf{Step 3. For a generic pair $(U,V)$, the restricted bilinear form has rank $2$.}

We claim that there exists a nonempty Zariski open subset
\[
\Omega_B\subset Gr_2(\mathbb{C}^{n_1})\times Gr_2(\mathbb{C}^{n_2})
\]
such that for every $(U,V)\in \Omega_B$, the restriction
\[
B|_{U\times V}
\]
has rank $2$.

To see this, let $U\subset \mathbb{C}^{n_1}$ and $V\subset \mathbb{C}^{n_2}$ be
$2$-planes, and choose bases $u_1,u_2$ of $U$ and $v_1,v_2$ of $V$. The bilinear form
$B|_{U\times V}$ is represented by the $2\times 2$ matrix
\[
M_B(U,V)=\bigl(B(u_i,v_j)\bigr)_{1\le i,j\le 2}.
\]
The condition that $B|_{U\times V}$ have rank $<2$ is equivalent to
\[
\det M_B(U,V)=0.
\]
Again, this is an algebraic condition on $(U,V)$, hence defines a Zariski closed subset
of the product Grassmannian.

To prove that this locus is proper, it is enough to produce one pair $(U,V)$ for which
the determinant is nonzero. Take
\[
U_0:=\mbox{ The vector subspace spanned by }(e_1,e_2)\subset \mathbb{C}^{n_1},\qquad
V_0:=\mbox{ The vector subspace spanned by }(f_1,f_2)\subset \mathbb{C}^{n_2},
\]
where $e_i$ and $f_j$ are the standard basis vectors in the $x$-space and $y$-space,
respectively. Then
\[
B|_{U_0\times V_0}(u,v)=u_1v_1+u_2v_2,
\]
so the corresponding matrix is the identity matrix
\[
\begin{pmatrix}
1&0\\
0&1
\end{pmatrix},
\]
whose determinant is $1\neq 0$. Thus the bad locus is proper, and its complement
$\Omega_B$ is a nonempty Zariski open subset.

\medskip

\noindent
\textbf{Step 4. The generic set on which both good properties hold.}

Set
\[
\Omega:=\bigl(\Omega_Q\times Gr_2(\mathbb{C}^{n_2})\bigr)\cap \Omega_B.
\]
Since $\Omega_Q$ and $\Omega_B$ are nonempty Zariski open subsets, so is $\Omega$.
Therefore, for every $(U,V)\in \Omega$, the following two properties hold simultaneously:

\begin{enumerate}
\item the quadratic form $Q|_U$ is nondegenerate;
\item the bilinear form $B|_{U\times V}$ has rank $2$.
\end{enumerate}

\medskip

\noindent
\textbf{Step 5. Normal form of the restriction.}

Fix $(U,V)\in \Omega$. We now show that $F|_{U\times V}$ is linearly equivalent to
\[
x^2+y^2+1-2ax-2by.
\]

Choose linear coordinates $(u_1,u_2)$ on $U$ and $(v_1,v_2)$ on $V$. Since $Q|_U$ is
nondegenerate, there exists an invertible linear change of coordinates on $U$ such that,
in the new coordinates (which we continue to denote by $(u_1,u_2)$),
\[
Q|_U(u)=u_1^2+u_2^2.
\]
Under this change of coordinates, the bilinear form $B|_{U\times V}$ is still a bilinear
form of rank $2$, and hence can be written as
\[
B|_{U\times V}(u,v)=u^T M v
\]
for some invertible $2\times 2$ matrix $M$.

Now perform a linear change of coordinates on $V$ by setting
\[
w:=Mv.
\]
Since $M$ is invertible, this is an invertible linear coordinate change on $V$.
In the coordinates $w=(w_1,w_2)$, we have
\[
u^TMv=u^Tw=u_1w_1+u_2w_2.
\]
Therefore, in the coordinates $(u_1,u_2)$ on $U$ and $(w_1,w_2)$ on $V$,
\[
F|_{U\times V}(u,w)
=
u_1^2+u_2^2+1-2(u_1w_1+u_2w_2).
\]
Renaming the variables by
\[
x:=u_1,\qquad y:=u_2,\qquad a:=w_1,\qquad b:=w_2,
\]
we obtain
\[
F|_{U\times V}\sim x^2+y^2+1-2ax-2by.
\]
This proves the asserted linear equivalence.

\medskip

\noindent
\textbf{Step 6. Sharpness of the exponent $4/3$ for the generic restrictions.}

Let
\[
f(x,y,a,b):=x^2+y^2+1-2ax-2by.
\]
By the known sharp example in the $\mathbb{C}^2\times \mathbb{C}^2$ setting, there
exist finite sets
\[
P_0,Q_0\subset \mathbb{C}^2,\qquad |P_0|\sim |Q_0|\sim n,
\]
such that
\[
|Z(f)\cap (P_0\times Q_0)|\sim n^{4/3}.
\]

Now fix $(U,V)\in \Omega$. By Step~5, there exist invertible linear maps
\[
T_U:U\to \mathbb{C}^2,\qquad T_V:V\to \mathbb{C}^2,
\]
such that
\[
F|_{U\times V}(u,v)=0
\quad\Longleftrightarrow\quad
f\bigl(T_U(u),T_V(v)\bigr)=0.
\]
Define
\[
P:=T_U^{-1}(P_0)\subset U,\qquad Q:=T_V^{-1}(Q_0)\subset V.
\]
Then
\[
|P|=|P_0|\sim n,\qquad |Q|=|Q_0|\sim n,
\]
and the bijectivity of $T_U$ and $T_V$ implies
\[
\bigl|Z(F|_{U\times V})\cap (P\times Q)\bigr|
=
|Z(f)\cap (P_0\times Q_0)|
\sim n^{4/3}.
\]
This proves that the exponent $4/3$ is sharp for every restriction corresponding to a
pair $(U,V)\in \Omega$.

\medskip

\noindent
\textbf{Step 7. Proof of consequence (i).} 

We argue as in the proof of Theorem \ref{Theorem2}. For every $P_1\subset \mathbb{C}^{n_1}$ and $Q_1\subset \mathbb{C}^{n_2}$ effective algebraic cycles of pure codimension $2$ and $\deg (P_1)=\deg (Q_1)=n$, let $W_1$ be the union of components of $P_1\times Q_1$ which belong to the set $\{F=0\}$, then we want to show that 
$$\deg (W_1)\lesssim n^{4/3}.$$

We choose a generic $U\times V$ in $Gr_2(\mathbb{C}^{n_1})\times Gr_2(\mathbb{C}^{n_2})$. Then as seen in previous steps, $F_{U\times V}$ is a non-Cartesian polynomial of degree $2$. 

By B\'ezout's theorem, $P=P_1|_U$ and $Q=Q_1|_V$ are finite sets, with $|P_1|=|Q_1|=n$. Similarly, $W=W_1|_{U\times V}$ is a finite set with $|W|=\deg (W_1)$. 

Since the constant in the intersection estimate for non-Cartesian polynomials on $\mathbb{C}^2\times \mathbb{C}^2$: 
$$|W|\lesssim n^{4/3}$$
depends only on the degree of the polynomial, the estimate we obtain here is independent of the choice of $U$ and $V$. 

\medskip

\noindent
\textbf{Step 8. Proof of consequence (ii).}  

Here we show the existence of $P_1,Q_1$ which realise the upper bound for $|W_1|$ in (i). To this end, we choose $P_0,Q_0$ as in Step 6, and for each $U\in Gr_2(\mathbb{C}^{n_1})$ and $V\in Gr_2(\mathbb{C}^{n_2})$ define $P_U=T_U^{-1}(P_0)$ and $Q_U=T_V^{-1}(Q_0)$ again as in Step 6. It is clear that we can arrange so that the map $(U,V)\in \Omega \mapsto (T_U,T_V)$ is smooth, and $\Omega $ is a product of Zariski open dense sets in $Gr_2(\mathbb{C}^{n_1)})$ and $Gr_2(\mathbb{C}^{n_2)}$. We can represent this as follows: 

Let $E=\{(U,V,u,v)\in Gr_2(\mathbb{C}^{n_1})\times Gr_2(\mathbb{C}^{n_2})\times \mathbb{C}^{n_1}\times \mathbb{C}^{n_2}: u\in U, v\in V\}$ (the universal bundle), and $\widehat{\Omega}=\{(U,V,u,v)\in \Omega \times \mathbb{C}^{n_1}\times \mathbb{C}^{n_2}: u\in U, v\in V\}$, then we have a smooth algebraic map $\mathcal{T}: \widehat{\Omega}\rightarrow \mathbb{C}^2$ such that $\mathcal{T}(U,V,u,v)=(T_Uu,T_Vv)$.

We let $\mathcal{R}=$ the closure of $\mathcal{T}^{-1}(P_0\times Q_0)$ in $E$. By definition: $\mathcal{R}\cap \widehat{\Omega}=\{(U,V,u,v): (U,V)\in \Omega,  u\in P_U, v\in Q_V\}$. It is easy to check that $\mathcal{R}=\mathcal{P}\times \mathcal{Q}$, for $\mathcal{P}\subset Gr_2(\mathbb{C}^{n_1})\times \mathbb{C}^{n_1}$ and $\mathcal{Q}=Gr_2(\mathbb{C}^{n_2})\times \mathbb{C}^{n_2}$, both having pure codimension $2$ and  
$$\deg (\mathcal{P})=|P_0|=n, \deg (\mathcal{Q})=|Q_0|=n.$$

Let $P_1\times Q_1$ be the image of $\mathcal{P}\times \mathcal{Q}$ under the natural projection $\pi :(Gr_2(\mathbb{C}^{n_1})\times \mathbb{C}^{n_1})\times (Gr_2(\mathbb{C}^{n_2})\times \mathbb{C}^{n_2})\rightarrow (\mathbb{C}^{n_1}\times \mathbb{C}^{n_2})$. Since $\pi$ is flat, it follows that $P_1$ and $Q_1$ are both of codimension $2$, and $\deg (P_1)=\deg (\mathcal{P})=|P_0|=n$, $\deg (Q_1)=\deg (\mathcal{Q})=|Q_0|=n$. Also, it is easy to check that for $(U,V)\in \Omega$, we have $P_1|_{U}=P_U$ and $Q_1|_V=Q_V$. 

Choose generic such $U,V$. Since $P_U\times Q_U$ realises the sharp intersection estimate for $F|_{U\times V}$, by B\'ezout's theorem we obtain that $P_1\times Q_1$ realises the sharp intersection estimate for $F$. 

The proof is complete.
\end{proof}

\begin{remark}
The statement above is a \emph{generic restriction} result. It does \emph{not} claim that
for every pair of $2$-planes $(U,V)$, the restriction $F|_{U\times V}$ is linearly
equivalent to the standard sharp model. For special choices of $U$ or $(U,V)$, the
restricted quadratic form may become degenerate, or the restricted bilinear form may drop
rank. The proposition asserts only that these bad phenomena occur on a proper algebraic
subset of
\[
Gr_2(\mathbb{C}^{n_1})\times Gr_2(\mathbb{C}^{n_2}).
\]
\end{remark}

\begin{remark}
If $n_2<n_1$, one may instead define
\[
F(x,y)=x_1^2+\cdots+x_{n_1}^2+1-2(x_1y_1+\cdots+x_{n_2}y_{n_2}),
\]
or, more invariantly, write
\[
F(x,y)=Q(x)+1-2B(x,y),
\]
where $Q$ is any nondegenerate quadratic form on $\mathbb{C}^{n_1}$ and $B$ is any
bilinear form on $\mathbb{C}^{n_1}\times \mathbb{C}^{n_2}$ having rank at least $2$.
The same proof shows that for a Zariski generic pair $(U,V)$ of $2$-planes, the
restriction $F|_{U\times V}$ is linearly equivalent to
\[
x^2+y^2+1-2ax-2by.
\]
\end{remark}

\subsubsection{An example on $\mathbb{C}^{n_1}\times \mathbb{C}^{n_2}$, of lower dimensional nature} In the spirit of Proposition \ref{prop:3block_lifting_lower}, one may ask whether a result similar to Proposition \ref{PropositionHigherDimensionalExampleGenericSliceGood} can be achieved, with a polynomial $F$ of a lower dimensional nature. Here we present one such example, whose proof is similar to that of  Proposition \ref{PropositionHigherDimensionalExampleGenericSliceGood} and hence is omitted. 

\begin{proposition}
Let $n_1,n_2\ge 2$. Let
\[
\pi:\mathbb{C}^{n_1}\to \mathbb{C}^2,
\qquad
\sigma:\mathbb{C}^{n_2}\to \mathbb{C}^2
\]
be surjective linear maps. Let
\[
f_0(u,v,a,b)=u^2+v^2+1-2au-2bv
\]
be the standard polynomial on $\mathbb{C}^2\times \mathbb{C}^2$.

Define
\[
F(x,y):=f_0\bigl(\pi(x),\sigma(y)\bigr).
\]

Then there exists a nonempty Zariski open subset
\[
\Omega \subset
Gr_2(\mathbb C^{n_1})
\times
Gr_2(\mathbb C^{n_2})
\]
such that for every $(U,V)\in \Omega$,  the sharp incidence exponent $4/3$ holds for
$F|_{U\times V}$ in the sense that there exist finite sets
\[
P\subset U,\qquad Q\subset V,\qquad |P|\sim |Q|\sim n,
\]
such that
\[
\bigl| Z(F|_{U\times V})\cap (P\times Q)\bigr| \sim n^{4/3}.
\]

As a consequence:  

i) For every $P_1\subset \mathbb{C}^{n_1}$ and $Q_1\subset \mathbb{C}^{n_2}$ effective algebraic cycles of pure codimension $2$ and $\deg (P_1)=\deg (Q_1)=n$, let $W_1$ be the union of components of $P_1\times Q_1$ which belong to the set $\{F=0\}$, then 
$$\deg (W_1)\lesssim n^{4/3}.$$

ii) There are $P_1\subset \mathbb{C}^{n_1}$ and $Q_1\subset \mathbb{C}^{n_2}$ effective algebraic cycles of pure codimension $2$ and $\deg (P_1)=\deg (Q_1)=n$, such that if $W_1$ is the union of components of $P_1\times Q_1$ which belong to the set $\{F=0\}$, then 
$$\deg (W_1)\sim n^{4/3}.$$
\end{proposition}

\section{Acknowledgments}

 The first author was supported by the National Science and Technology Council (NSTC) under Grant No.~111-2115-M-002-010-MY5. This paper has been initiated while the second author was visiting National Taiwan University in August 2025. He would like to thank the instution for hospitality and support and also thanks University of Oslo for support. He would also like to thank Keiji Oguiso and his students for some inspiring discussions on this topic.


\begin{thebibliography}{99}

\bibitem{Barakat-Hegermann}
M. Barakat and M. Lange-Hegermann,
\newblock An algorithmic approach to Chevalley's theorem on images of rational morphisms between affine varieties,
\newblock \emph{Mathematics of Computation}, 91(333):451--490, 2021.

\bibitem{Brownawell}
W.~D.~Brownawell,
\newblock Bounds for the degrees in the Nullstellensatz,
\newblock \emph{Annals of Mathematics}, 126(3):577--591, 1987.

\bibitem{BukhTsimerman}
B.~Bukh and J.~Tsimerman,
\newblock Sum-product estimates for rational functions,
\newblock \emph{Proceedings of the London Mathematical Society}, 104(1):1--26, 2012.

\bibitem{CoxLittleOShea}
D.~Cox, J.~Little, and D.~O'Shea,
\newblock \emph{Ideals, Varieties, and Algorithms: An Introduction to Computational Algebraic Geometry and Commutative Algebra},
\newblock Springer, New York, 4th edition, 2015.

\bibitem{Eisenbud}
D.~Eisenbud,
\newblock \emph{Commutative Algebra with a View Toward Algebraic Geometry},
\newblock Springer, New York, 1995.

\bibitem{ElekesRonyai}
G.~Elekes and L.~R\'onyai,
\newblock A combinatorial problem on polynomials and rational functions,
\newblock \emph{Journal of Combinatorial Theory, Series A}, 89(1):1--20, 2000.



\bibitem{GuthKatz}
L.~Guth and N.~H.~Katz,
\newblock On the Erd\H{o}s distinct distances problem in the plane,
\newblock \emph{Annals of Mathematics}, 181(1):155--190, 2015.

\bibitem{Hartshorne}
R.~Hartshorne,
\newblock \emph{Algebraic Geometry},
\newblock Springer, New York, 1977.

\bibitem{KaplanMatousekSharir}
H.~Kaplan, J.~Matou\v{s}ek, and M.~Sharir,
\newblock Simple proofs of classical theorems in discrete geometry via the Guth--Katz polynomial partitioning technique,
\newblock \emph{Discrete \& Computational Geometry}, 48(3):499--517, 2012.

\bibitem{Kollar}
J.~Koll\'ar,
\newblock Sharp effective Nullstellensatz,
\newblock \emph{Journal of the American Mathematical Society}, 1(4):963--975, 1988.

\bibitem{MojarradPhamValculescuZeeuw}
H.~N.~Mojarrad, T.~Pham, C.~Valculescu, and F.~de~Zeeuw,
\newblock Schwartz--Zippel bounds for two-dimensional products,
\newblock \emph{Discrete Analysis}, 2017:20, 1--20, 2017.

\bibitem{PachSharir}
J.~Pach and M.~Sharir,
\newblock On the number of incidences between points and curves,
\newblock \emph{Combinatorics, Probability and Computing}, 7(1):121--127, 1998.

\bibitem{RazSharirZeeuw}
O.~E.~Raz, M.~Sharir, and F.~de~Zeeuw,
\newblock Polynomials vanishing on Cartesian products: The Elekes--Szab\'o theorem revisited,
\newblock \emph{Duke Mathematical Journal}, 165(18):3517--3566, 2016.

\bibitem{SharirSolomon}
M.~Sharir and N.~Solomon,
\newblock Incidences between points and lines in $\mathbb{R}^4$,
\newblock \emph{Discrete \& Computational Geometry}, 57(3):702--756, 2017.

\bibitem{SharirZahl}
M.~Sharir and J.~Zahl,
\newblock Cutting algebraic curves into pseudo-segments and applications,
\newblock \emph{Journal of Combinatorial Theory, Series A}, 150:1--35, 2017.

\bibitem{SolymosiTao}
J.~Solymosi and T.~Tao,
\newblock An incidence theorem in higher dimensions,
\newblock \emph{Discrete \& Computational Geometry}, 48(2):255--280, 2012.

\bibitem{SolymosiZeeuw}
J.~Solymosi and F.~de~Zeeuw,
\newblock Incidence bounds for complex algebraic curves on Cartesian products,
\newblock in \emph{New Trends in Intuitive Geometry},
\newblock Bolyai Society Mathematical Studies, vol.~27, Springer, Berlin, 2018, pp.~385--405.

\bibitem{Sombra}
M.~Sombra,
\newblock A sparse effective nullstellensatz,
\newblock in \emph{Advances in Applied Mathematics}, 22 (2): 271--295, 1999. 

\bibitem{TaoVu}
T.~Tao and V.~Vu,
\newblock \emph{Additive Combinatorics},
\newblock Cambridge University Press, Cambridge, 2006.

\bibitem{Zahl}
J.~Zahl,
\newblock An improved bound on the number of point-surface incidences in three dimensions,
\newblock \emph{Contributions to Discrete Mathematics}, 8(1):100--121, 2013.

\end{thebibliography}
\end{document}